\documentclass[11pt]{amsart}
\usepackage{amssymb,latexsym,amsfonts}
\usepackage{amsmath}
\usepackage{amscd}
\usepackage{pb-diagram,pb-xy}
\usepackage{tikz}
\usetikzlibrary{shapes,arrows}
\usepackage{graphicx}
\usepackage{hyperref}
\usepackage{color}
\usepackage[matrix,arrow]{xy}

\usepackage{fullpage}
\newcommand{\str}{{\star}}

\newcommand{\Weit}{{\mathbf {Weit}}}

\newcommand{\Ker}{{\rm Ker}}
\newcommand{\cc}{{\mathbf {c}} }
\newcommand{\tD}{{\widetilde {D}} }
\newcommand{\td}{{\tilde {\delta}} }

\newcommand{\Ric}{{\rm Ric}}

\DeclareMathAlphabet{\mathpzc}{OT1}{pzc}{m}{it}

\def\Min{\, \mathbf {min}}

\newtheorem{theorem}{Theorem}[section]
\newtheorem{corollary}[theorem]{Corollary}

\newtheorem{lemma}[theorem]{Lemma}

\newtheorem{proposition}[theorem]{Proposition}
\newtheorem*{thm-a}{Theorem\! A}
\newtheorem*{thm-aa}{Theorem\! ${\rm A}^{\prime}$}
\newtheorem*{cor-bb}{Corollary\! ${\rm B}^{\prime}$}
\newtheorem*{thm-cc}{Theorem\! ${\rm C}^{\prime}$}
\newtheorem*{cor-a}{Corollary\! A}
\newtheorem*{thm-b}{Theorem\! B}
\newtheorem*{cor-b}{Corollary\! B}
\newtheorem*{thm-c}{Theorem\! C}
\newtheorem*{cor-c}{Corollary\! C}
\newtheorem*{thm-d}{Theorem\! D}
\newtheorem*{thm-dd}{Theorem\! ${\rm D}^{\prime}$}
\newtheorem*{cor-dd}{Corollary\! ${\rm D}^{\prime}$}
\newtheorem*{cor-d}{Corollary\! D}
\newtheorem*{thm-e}{Theorem\! E}
\newtheorem*{cor-e}{Corollary\! E}
\newtheorem*{thm-f}{Theorem\! F}
\newtheorem*{cor-f}{Corollary\! F}

\newtheorem*{conj-d}{Conjecture D}
\newtheorem*{conj-c}{Conjecture C}
\newtheorem{thm}{Theorem}[section]
\theoremstyle{definition}
\newtheorem{definition}[thm]{Definition}

\newtheorem*{remark}{Remark}

\begin{document}
\author{Mohammed Larbi Labbi} 
\address{Department of Mathematics\\
 College of Science\\
  University of Bahrain\\
  32038, Bahrain.}
\email{mlabbi@uob.edu.bh}
\renewcommand{\subjclassname}{
  \textup{2020} Mathematics Subject Classification}

\title{Hodge-de Rham and Lichn\'erowicz Laplacians on double forms and some vanishing theorems}  \maketitle
 \begin{abstract}
A $(p,q)$-double form on a Riemannian manifold $(M,g)$ can be considered simultaneously as a vector-valued differential $p$-form over $M$ or alternatively as a vector-valued $q$-form. Accordingly, the usual Hodge-de Rham Laplacian on differential forms can be extended to  double forms in two ways. The differential operators obtained in this way are denoted by $\Delta$ and $\widetilde{\Delta}$.\\
In this paper, we show that the Lichn\'erowicz Laplacian $\Delta_L$ once operating on double forms, is nothing but the average of the two operators mentioned above. We introduce a new  product on double forms to establish index-free formulas for the curvature terms in the Weitzenb\"ock formulas corresponding to the Laplacians $\Delta, \widetilde{\Delta}$ and $\Delta_L$. We prove vanishing theorems for the Hodge-de Rham Laplacian $\Delta$  on $(p,0)$ double forms and for $\Delta_L$ and $\Delta$ on symmetric double forms of arbitrary order. These results generalize recent results by Petersen-Wink. Our vanishing theorems  reveal  the impact of the role played by the rank of the eigenvectors of the curvature operator on the structure (e.g. the topology) of the manifold.

\end{abstract}

\tableofcontents

\section{Summary of the main results}\label{intro}
A $(p,q)$-double form $\omega$ on a Riemannian manifold $(M,g)$ is a covariant tensor of rank $p+q$ which is alternating with respect to the first $p$ and the last $q$-arguments. The double form $\omega$ can be regarded  as a vector valued  differential $p$-form over $M$ or alternatively as a vector valued  differential $p$-form over $M$. Accordingly,  the exterior derivative $d$, divergence $\delta$ and  Hodge de Rham Laplacian $\Delta$  of differential forms can be extended in two ways  to  operators that act on double forms. These extended operators are denoted  $D,\delta, \Delta$ and $\tD, \td, \widetilde{\Delta}$ respectively. The operators thus obtained on double forms were first introduced and studied by several authors, see e.g.   Nomizu \cite{Nomizu}, Gray\cite{Gray}, Kulkarni \cite{Kulkarni}, Bourguignon \cite{Bourguignon}, Nasu \cite{Nasu} Ogawa \cite{Ogawa}, ...\\
In \cite{Bourguignon}, Bourguignon has established two index-free formulations of the curvature term in Weitzenb\"ock's formula for the Hodge de Rham Laplacian $\Delta =D\delta+\delta D$ on $(1,1)$ and $(2,2)$ double forms.  A first result in this paper is a generalization of Bourguignon's two formulas to general double forms. For a $(p,q)$ double form $\omega$ we show that
\begin{equation}\label{pqWeit}
(D\delta+\delta D)\omega=\nabla^*\nabla+\omega \circ \Bigl( \frac{g^{p-1}\Ric}{(p-1)!}-2 \frac{g^{p-2}R}{(p-2!} \Bigr)-\frac{1}{4} R\# \omega,
\end{equation}
where $p,q\geq 0$, $g^0=1$ and $g^k=0$ if $k$ is negative. The binary operation $\circ$ denotes the composition product of double forms, the symbol free products are exterior products of double forms. The sharp operation $\#$ is a new binary operation on double forms which we  introduce here in this paper using the Clifford product of forms. It  is in fact a  generalization of the sharp product of $(2,2)$ double forms that appears in the evolution equation of the Riemann curvature tensor under the Ricci flow, see  \cite{BW} and \cite{Shmidt}.\\
For the case $q=0$, the term $R\# \omega$ vanishes. The curvature term in the above equation \ref{pqWeit} provides an index-free  formula for the curvature term in the Weitzenb\"ock formula for differential forms which was first proved in \cite{Bourguignon, Labbi-nachr}.\\
In this paper we prove similar formulas  for the adjoint de Rham Laplacian $ \widetilde{\Delta}=\tD\td+\td\tD$ and the Lichn\'erowicz Laplacian $\Delta_L$ both acting on general $(p,q)$ double forms.
We show that the Laplacian
 $\Delta_L$ acting on double forms coincides with the average of the two Laplacians $\Delta$ and $\widetilde{\Delta}$. More precisely, we prove that
\begin{equation}
\Delta_L=\frac{1}{2}\bigl( \Delta +\tilde{\Delta}\bigr)=\frac{1}{2}\bigl( D\delta+\delta D+\tilde{D}\tilde{\delta}+\tilde{\delta}\tilde{D}.\bigr)
\end{equation}

Weitzenb\"ock's formulas show that the principal part of the operators $\Delta, \tilde{\Delta}$ and $\Delta_L$ is the connection Laplacian $\nabla^*\nabla$, hence they are elliptic operators. Therefore they have finite-dimensional kernels if $M$ is a compact manifold. We call the kernel of $\Delta$ (resp. $\tilde{\Delta}$, $\Delta_L$) the space of harmonic (resp. $\tilde{\Delta}$-harmonic, $\Delta_L$-harmonic)  double forms.\\
It turns out that   the kernels of $\Delta$, $\tilde{\Delta}$ and $\Delta_L$ all coincide on symmetric or skew-symmetric double forms on a compact manifold. In general one has only $\Ker(\Delta_L)\subset \Ker(\Delta)$, since for example, a harmonic $(p,0)$ double form is  $\Delta_L$-harmonic   if and only if it is parallel.\\
In the second part, we prove several vanishing theorems  for  the kernels of $\Delta$ and $\Delta_L$.\\
To illustrate our approach, we first recall the classical Weitzenb\"ock formula for a differential $1$-form $\omega$
\[\Delta \omega=\nabla^*\nabla\omega+\Ric(\omega).\]
Therefore, the positivity of the Ricci curvature implies the vanishing of all harmonic $1$-forms. The Ricci curvature is related to the eigenvalues and eigenvectors of the Riemann curvature operator through the following nice formula, see formula (\ref{2-identiries}) in the next section.
\begin{equation}\label{RicRiem}
\Ric=-\sum_{\alpha} \lambda_{\alpha}\, \rho(E_{\alpha})\circ \rho(E_{\alpha}),
\end{equation}
where, $(\lambda_{\alpha})$ are the eigenvalues of the Riemann curvature operator and $(E_{\alpha})$ is an orthonormal basis of its  eigenvectors at the point under consideration. The  mapping $\rho$ is linear and sends $2$-forms onto $(1,1)$ skew symmetric double forms. By definition,  we define  $\rho(e^1\wedge e^2)=e^1\otimes e^2-e^2\otimes e^1$ for  co-vectors $e^1$ and $e^2$.  This is the inverse of the alt operator. The operation $\circ$ is the composition product of double forms.\\
We note that the bilinear form $-\rho(E_{\alpha})\circ \rho(E_{\alpha})$ is non-negative for all $\alpha$. The following  algebraic key lemma is a 
reformulation of an inequality of Petersen-Wink \cite{Pet-Wink}  
\begin{lemma}\label{first-lemma}
Let $\lambda_1\leq \lambda_2\leq ...\leq \lambda_n$ and $w_1, ...,w_n$ be given real numbers such that $w_i\geq 0$ for all $i$ and $M=\max\{w_1,...,w_n\}>0$. Let $S=\sum_{i=1}^nw_i$, then
\[ \lfloor \frac{S}{M}\rfloor \sum_{i=1}^n\lambda_iw_i\geq S\sum_{i=1}^{\lfloor \frac{S}{M}\rfloor}\lambda_i.\]
\end{lemma}

In the above case of Ricci curvature, we have $S$ equals  the Ricci curvature of the standard sphere, which is $n-1$ and $M=1$.
Therefore the Ricci curvature is positive under the condition that the  sum of the lowest $(n-1)$ eigenvalues of the Riemann tensor is positive.\\

We will apply a similar analysis to the general case of double forms. Since the metric $g$ is fixed on $M$, we will identify whenever convenient spaces and their duals, and operators with the corresponding bilinear forms. Let $E$ be a given $2$-form and $\rho(E)$ the corresponding $(1,1)$ double form as above. First, we extend $\rho(E)$ by derivations to a double form on the diagonal exterior algebra of double forms and then we extend by derivations the corresponding operator on forms  to operate on the tensor product algebra that is on the exterior algebra of double forms. See section \ref{extensions-section} for more details about these extensions. We shall denote by 
 $\Bigl(\rho(E)\Bigr)_d$ the  extension thus obtained. This is nothing else than the standard representation of the Lie algebra $\mathfrak{so}(n)$  in the exterior algebra of double forms.  We prove that for a given $(p,q)$ double form $\omega$ one has
\begin{equation}
\Bigl(\rho(E)\Bigr)_d(\omega)=\frac{g^{q-1}\rho(E)}{(q-1)!}\circ \omega -\omega \circ \frac{g^{p-1}\rho(E)}{(p-1)!}.
\end{equation}
\medskip\noindent
The curvature term $\Ric_L$ defining the Lichn\'erowicz Laplacian was first defined by Lichn\'erowicz in  \cite[page 27]{Lichn}. 
It can  alternatively be defined by (using the same notation as in the above formula (\ref{RicRiem}) for the Ricci curvature), see  \cite[page 54]{Besse} and \cite[page 345]{Petersen-book}
\begin{equation}
\Ric_L(\omega)=-\sum_{\alpha} \lambda_{\alpha}\,\, \Bigl(\rho(E_{\alpha})\Bigr)_d\circ \Bigl(\rho(E_{\alpha})\Bigr)_d(\omega).
\end{equation}
We recover the Ricci curvature if $\omega$ is a $(1,0)$ double form, that is a $1$-form. Recall that for $p$-forms seen as $(p,0)$ double forms, the Hodge de Rham Laplacian satisfies
\[\Delta \omega=\nabla^*\nabla\omega+\Ric_L(\omega).\]
For a general $(p,q)$ double form $\omega$, the Lichn\'erowicz Laplacian is defined by
\[\Delta_L \omega=\nabla^*\nabla\omega+\frac{1}{2}\Ric_L(\omega).\]
Therefore, the study of the positivity of $\Ric_L$ is important since it provides in particular vanishing theorems for the kernel of $\Delta$ on differential forms  and the kernel of $\Delta_L$ on differential double forms.\\
\medskip
We begin our study of the Lichn\'erowicz curvature term  by proving the following (universal) reformulation of $\Ric_L$ in terms of the above three operations on double forms
\begin{equation}
\Ric_L(\omega)=\frac{1}{2}\Bigl(\omega \circ \mathcal{N}_p+\mathcal{N}_q\circ \omega\Bigr) -\frac{1}{2}R\# \omega,
\end{equation}
where $\omega$ is a $(p,q)$ double form, $\mathcal{N}_k=\frac{g^{k-1}\Ric}{(k-1)!}-2 \frac{g^{k-2}R}{(k-2)!}$ for $k\geq 2$ and $\mathcal{N}_1=\Ric, \mathcal{N}_0=0. $\\
For $q=0$ one has $R\# \omega=0$ and then we recover the formula for the Weitzenb\"ock curvature operators  as in \cite{Bourguignon, Labbi-nachr}.\\
\medskip
\noindent
Recall that the rank of a 2-form $\alpha$ is the minimum number of co-vectors  in terms of which $\alpha$ can be expressed. We use the above lemma \ref{first-lemma} and some analysis to prove the following important proposition which provides sufficient conditions that guarantee the positivity of the $\Ric_L$-curvature

\begin{proposition}
Let $x\in M$ be fixed and let $(E_{\alpha})$ be an orthonormal basis of $2$-forms  diagonalizing the Riemann curvature operator at $x$. We assume that  for each  $\alpha$, the rank of $E_{\alpha}$ is $\leq 2r$, for some fixed integer $r$.
\begin{itemize}
\item[a)] Let $\omega$ be a $(p,0)$ double form with $p\geq 1$. If the sum of the lowest $\lfloor \frac{p(n-p)}{\Min\{p,r\}}\rfloor$ eigenvalues (counted with multiplicity) of the Riemann curvature  tensor at $x$ is positive (resp. non-negative) then 
$\langle \Ric_L(\omega),\omega\rangle >0$  (resp.  $\geq 0$).

\item[b)]  Let $\omega$ be a symmetric and trace free $(p,p)$ double form that satisfies the first Bianchi identity with $p\geq 1$. If the sum of the lowest $\lfloor \frac{p(n-p+1)}{2\Min\{p,r\}}\rfloor$ eigenvalues (counted with multiplicity) of the Riemann curvature  tensor at $x$ is positive (resp. non-negative) then 
$\langle \Ric_L(\omega),\omega\rangle >0$  (resp.  $\geq 0$).
\end{itemize}
\end{proposition}
Before we state the main theorems, we need a definition
\begin{definition}
\begin{itemize}
\item[a)] We say that the Riemann curvature tensor $R$  of a Riemannian $n$-manifold has {\emph purity rank} $\leq 2r$, for some $r$ with $2\leq 2r\leq n$,  if for every $m\in M$,  there exists  an orthonormal basis of eigenvectors of the curvature operator $R$ each of which has rank $\leq 2r$.
\item[b)] We say that the Riemann curvature tensor $R$ is $k$-positive (resp. $k$-non-negative) if the sum of the lowest $k$ eigenvalues of the associated curvature operator (counted with multiplicity) is positive (resp. non-negative).
\end{itemize}
\end{definition}
The Riemann tensor has purity rank $\leq 2$ if and only if $R$ is pure in the sense of Maillot \cite{Maillot-a, Maillot-t}. The standard round sphere $S^n$, $n\geq 2$, has purity rank $\leq 2$ and the complex (resp. quaternionic) projective space $\mathbb{C}P^n$ (resp. $\mathbb{H}P^n$)  with the Fubiny-Study metric has purity rank $\leq 4$ (resp. $\leq 8$) for $n\geq 2$. We note that the property of having purity rank $\leq 2k$, for  fixed $k$, is stable under Riemannian products.

\begin{thm-a}
Let $(M,g)$ be a closed and connected Riemannian manifold of dimension $n\geq 3$ and $p$ be a positive integer such that $2p\leq n$. Suppose the Riemann curvature tensor $R$ has purity rank $\leq 2r$ and $R$ is $k$-positive (resp. $k$-nonnegative) with
 $1\leq k\leq   \frac{p(n-p)}{\Min\{p,r\}}$,
then the Betti numbers $b_p$ and $b_{n-p}$ of $M$ vanish (resp. every harmonic $p$-form on $M$ is parallel).
\end{thm-a}
Note that in the case of minimal purity where $r=1$, the minimum of $p$ and $r$ is $1$ and we therefore  recover a result of Maillot \cite{Maillot-a, Maillot-t}, namely, the vanishing of $b_p$ under the condition that $R$ is $p(n-p)$-positive. On the other hand, we note that $n-p\leq \frac{p(n-p)}{\Min\{p,r\}}$, so we obtain the vanishing of $b_p$ under the condition that $R$ is $(n-p)$-positive.
We recover therefore a recent result of  Petersen-Wink \cite[Theorem A]{Pet-Wink} see also \cite[Theorem 3.15]{Bet-Good}.\\
\medskip
\begin{thm-b}
Let $(M,g)$ be a closed and connected Riemannian manifold of dimension $n\geq 3$ and let $\omega$ be a harmonic $(p,p)$ symmetric double form on $M$ with $1\leq p\leq n/2$,  satisfying the first Bianchi identity. 
 Let $j$ be an integer such that $1\leq j\leq p$.
 If the Riemann curvature tensor $R$ is $\lfloor \frac{n-j+1}{2}\rfloor$-positive then  $\cc^{p-j}\omega =\lambda g^j$ for some constant $\lambda$. In particular,
\begin{itemize}
\item[a)] If $R$ is $\lfloor \frac{n-p+1}{2}\rfloor$-positive then  $\omega$ has constant sectional curvature.
\item[b)] If $R$ is $\lfloor \frac{n}{2}\rfloor$-positive then  $\cc^{p-1}\omega =\lambda g$ for some constant $\lambda$.
\end{itemize}
\end{thm-b}
In the case where $p=2$ and $\omega=R$ is the Riemann tensor,  we recover the main result of Tachibana in \cite{Tachibana}. For $p=1$, we recover a well known result that a harmonic symmetric $(1,1)$ double form on a Riemannian manifold with positive curvature operator (indeed
 ${\rm sec}>0$ suffices) must be proportional to the metric tensor\\
A consequence of the above theorem is that a Riemannian $n$-manifold with  harmonic curvature and $\lfloor \frac{n}{2}\rfloor$-positive curvature tensor must be Einstein. More generally, if for some $1\leq k \leq n/2$,  the Gauss-Kronecker tensor $R^k$  is harmonic and the Riemann tensor is $\lfloor \frac{n-2k+1}{2}\rfloor$-positive (resp. $\lfloor \frac{n}{2}\rfloor$-positive) then the Riemannian manifold must be with constant Thorpe $2k$-sectional curvature (resp.  $2k$-Einstein in the sense of \cite{Labbi-variational}).\\
The following classes of Riemannian manifolds have their Gauss-Kronecker tensor $R^k$ harmonic and to which the above theorem applies.
\begin{enumerate}
\item $2k$-Thorpe manifolds \cite{Labbi-Balkan}, see also \cite{Kim, Labbi-JAUMS}. These are Riemannian manifolds of even dimension $n=2p\geq 4k$ that satisfy the self-duality condition
$\str\bigl(g^{p-2k}R^k\bigr)=g^{p-2k}R^k.$
These are generalizations of Einstein manifolds obtained for $k=1$
\item $2k$-anti-Thorpe manifolds. These are Riemannian manifolds of even dimension $n=2p\geq 4k$ that satisfy the  condition
$\str\bigl(g^{p-2k}R^k\bigr)+g^{p-2k}R^k=\lambda g^p,$
where $\lambda$ is a constant. These are generalizations of conformally flat manifolds with constant scalar curvature obtained for $k=1$, see \cite[Theorem 5.8]{Labbi-double-forms}.
\end{enumerate}

\medskip

\begin{thm-c}
Let $(M,g)$ be a closed and connected Riemannian manifold of dimension $n\geq 3$. Let $\omega$ be a harmonic $(p,p)$ symmetric double form on $M$ with $2\leq p\leq n/2$,  satisfying the first and second Bianchi identities such that its first contraction satisfies $\cc \omega=\lambda g^{p-1}$ for some function $\lambda$ on $M$.
Suppose the Riemann curvature tensor $R$ has purity rank $\leq 2r$ and $R$ is $k$-positive (resp. $k$-nonnegative) with
$1\leq k\leq \frac{p(n-p+1)}{2\Min\{p,r\}}$, then the double form $\omega$ has constant sectional curvature (resp. $\omega$ must be parallel). 
\end{thm-c}

\begin{remark}
We remark that a $(p,p)$ double form such that $2p>n$ and satisfying $\cc \omega=\lambda g^{p-1}$ must be with constant sectional curvature without the positivity and purity assumptions on $R$, see Propositions 2.1 and 2.2 in \cite{Labbi-Balkan}.
\end{remark}
\noindent
Recall that a Riemannian manifold is said to be hyper $2k$-Einstein if its Riemann curvature tensor seen as a $(2,2)$ double form  satisfies $\cc R^k= \lambda g^{2k-1}$, see \cite{Labbi-Balkan}. We recover the usual Einstein manifolds when $k=1$. The above theorem shows that 

\begin{cor-c}
Let $(M,g)$ be a closed and connected hyper $2k$-Einstein Riemannian $n$-manifold with $2\leq 2k\leq n/2$. Suppose the Riemann curvature tensor $R$ has purity rank $\leq 2r$ and let  $p$ be any integer such that $1\leq p\leq \frac{k(n-2k+1)}{\Min\{2k,r\}}$. 
If $R$ is $p$-positive (resp. $p$-nonnegative) then $M$ has constant $(2k)$-th Thorpe sectional curvature, that is $R^k={\mathrm constant}.g^{2k}$ (resp. $R^k$ must be parallel). 
\end{cor-c}
We remark that $\frac{n-2k+1}{2}\leq \frac{k(n-2k+1)}{\Min\{2k,r\}}$ for any value of $r$, therefore we have the same conclusion as in the above corollary under the assumption that the sum of the lowest $\lfloor \frac{n-2k+1}{2}\rfloor$ eigenvalues of the curvature operator $R$ is positive  (resp. nonnegative).
For $k=1$, we recover  a recent result of Petersen and Wink \cite[Theorem B]{Pet-Wink}.\\
\subsection*{Plan of the paper}
Section 2 is about preparatory materials concerning several algebraic properties of double forms. In section 3, we use the composition and exterior products of double forms to express different extensions by derivations of a vector space operator to an operator acting on the exterior algebras of forms and double forms. In section 4, we introduce the new sharp $\#$ operation on double forms and study some of its properties. In section 5, we study several differential properties of  double forms and use them to prove useful properties of the Hodge-de Rham Laplacian $\Delta$ on double forms. We use the $\#$ product, to establish a general index-free formula for the curvature  term in the Weitzenb\"ock formula for $\Delta$. Section 5 contains as well the proof of Theorem A. The proofs of Theorems B and C are included in section 6. In this last section we prove useful properties of the Lichn\'erowicz Laplacian $\Delta_L$ acting on double forms.

\section{Double forms: algebraic properties}

Let $(V,g)$ be an Euclidean real vector space  of finite dimension $n$.  Let
  $\Lambda V^{*}=\bigoplus_{p\geq 0}\Lambda^{p}V^{*}$   denotes the exterior algebra
 of the dual space  $V^* $.  \emph{The space of exterior  double forms} of $V$   is defined as
 $${\mathcal D}V^*:= \Lambda V^{*}\otimes \Lambda V^{*}=\bigoplus_{p,q\geq 0}
  {\mathcal D}^{p,q}V^*,$$
 where $  {\mathcal D}^{p,q}V^*= \Lambda^{p}V^{*} \otimes  \Lambda^{q}V^{*}.$
The space ${\mathcal D}V^*$ is naturally  a bi-graded associative  algebra, called \emph{double exterior algebra of $V^*$},
 where for simple elements  $\omega_1=\theta_1\otimes \theta_2\in { \mathcal D}^{p,q}V^*$ and
 $\omega_2=\theta_3\otimes \theta_4\in  {\mathcal D}^{r,s}V^*$, the multiplication is given  by
 \begin{equation}
 \label{def:prod}
 \omega_1.\omega_2= (\theta_1\otimes \theta_2 )(\theta_3\otimes
 \theta_4)=
 (\theta_1\wedge \theta_3 )\otimes(\theta_2\wedge \theta_4)\in
    {\mathcal D}^{p+r,q+s}V^*.\end{equation}
The wedge symbol $\wedge$ denotes as usual the exterior product in the exterior algebra $ \Lambda V^{*}$.
 The above multiplication in ${\mathcal D}V^*$  will be referred to as the {\em exterior product of double forms} and is denoted by a dot. We usually omit  this dot symbol if there will be no confusion. 
\subsection{Inner and interior products of double forms}
Let $\langle ., .\rangle$ denotes  the  standard inner product on $ \Lambda V^{*}$ and ${\mathcal D} V^{*}=\Lambda V^{*}\otimes \Lambda V^{*}$ both inherited from the Euclidean structure $(V,g)$ in the following standard way: If $(e_i)$ is an orthonormal basis of $V$ and $(e^i)$ is its dual basis, we declare that $\{e^{i_1}\wedge...\wedge e^{i_k}:i_1<...<i_k\}$ is an orthonormal basis of   $ \Lambda^k V^{*}$ and $\{e^{i_1}\wedge...\wedge e^{i_p}\otimes e^{j_1}\wedge...\wedge e^{j_q}:i_1<...<i_p, j_1<...<j_q\}$ is an orthonormal basis of   ${\mathcal D}^{p,q} V^{*}$. We set as well that the subspaces $ \Lambda^k V^{*}$ and $ \Lambda^l V^{*}$ are orthogonal if $k\not=l$ and the subspaces ${\mathcal D}^{p,q} V^{*}$, ${\mathcal D}^{r,s} V^{*}$ are orthogonal if $p\not= r$ or $q\not= s$. In particular, if $g^p$ denotes the $p$th exterior power of the $g$ in ${\mathcal D} V^{*}$ then  one has
$$\langle g^p, g^p\rangle=\frac{n!p!}{(n-p)!}\,\, {\rm and}\,\,\, \langle g^p, g^r\rangle=0\,\, {\rm if}\,\, p\not= r.$$
Next, for $\alpha\in \Lambda V^{*} $, {\em the interior product} map $\iota_{\alpha}:\Lambda V^{*}\to \Lambda V^{*}$ is by definition the adjoint of the left multiplication map by $\alpha$ that is
$\mu_{\alpha}: \Lambda V^{*}\to \Lambda V^{*}$ defined by $\mu_{\alpha}(\omega)=\alpha\wedge\omega$.\\
In a similar way, For a double form $\alpha\in {\mathcal D}V^*$,  we define {\em the interior product} map $\iota_{\alpha}: {\mathcal D}V^*\to {\mathcal D}V^*$ to be the adjoint of the left multiplication map by $\alpha$ that is
$\mu_{\alpha}: {\mathcal D} V^{*}\to {\mathcal D} V^{*}$ such that $\mu_{\alpha}(\omega)=\alpha.\omega$.\\
It is not difficult to verify that for simple double forms $\alpha=\alpha_1\otimes \alpha_2$ and $\omega=\omega_1\otimes \omega_2$ one has
\begin{equation}
\iota_{\alpha_1\otimes \alpha_2}=\iota_{\alpha_1}\omega_1\otimes \iota_{\alpha_2}\omega_2.
\end{equation}
For $\alpha=g$ seen as a $(1,1)$ double form, it turns out that the interior product $\iota_g\omega$ is nothing but the contraction  ${\cc}\omega$ of  the double form $\omega$.\\
We note that for a vector $v\in V$ and for the corresponding 1-form  $v^{\flat}=g(v,.)\in V^*$, one has for any $\omega\in \Lambda V^*$ the following
\begin{equation}
\iota_{(v^\flat)}\omega=\iota_v\omega=\omega(v, .).
\end{equation}
That is to say that the interior product $\iota_\alpha\omega$ coincides with the usual (insertion) interior product. The same is true for the interior product of double form, for $v_1,v_2\in V$ and for any $\omega=\omega_1\otimes \omega_2 \in {\mathcal {D}} V^*$ one has 
\begin{equation}
\iota_{(v_1^\flat\otimes v_2^\flat)}\omega=\iota_{(v_1\otimes v_2)}\omega=\iota_{v_1}\omega_1\otimes \iota_{v_2}\omega_2.
\end{equation}
For this reason we freely use both forms for the interior product in this paper. Especially, when we use an orthonormal  basis $(e_i)$ of $V$ and its dual basis $(e^i)$ as in the two formulas below, we shall most of the times in this paper write $\iota_{e_i}$ instead of $\iota_{e^i}$ and write $\iota_{e_i\otimes e_j}$ instead of $\iota_{e^i\otimes e^j}$. The following identities, whose proof is straightforward,  are stated here for future reference and are valid for any double form $\omega$
\begin{equation}\label{2identities}
\begin{split}
\iota_{e_i\otimes 1}\Bigl[ (e^j\otimes 1).\omega\Bigr]&=\delta_{ij}\omega-(e^j\otimes 1).\iota_{e_i\otimes 1}\omega,\\
\iota_{1\otimes e_i}\Bigl[ (1 \otimes e^j).\omega\Bigr]&=\delta_{ij}\omega-(1 \otimes e^j).\iota_{1\otimes  e_i}\omega,\\
\cc\bigl((e^i\otimes 1).\omega\bigr)&=\iota_{1\otimes e_i}\omega-(e^i\otimes 1).\cc \omega.
\end{split}
\end{equation}
Recall that $\cc=\iota_g$ is the contraction map of double forms.
\subsection{Hodge star operator acting on double forms}
Suppose that an orientation is fixed on $V$ and choose a positive unit vector $e\in\Lambda^nV^{*}$. Recall that for $\alpha\in \Lambda V^{*}$, the Hodge star operator is an isomorphism $\str: \Lambda V^{*}\to \Lambda V^{*}$  which can be defined using the inner product by
\[ \str \alpha =\iota_{\alpha}e.\]
It follows from the above definition that if $\beta\in \Lambda^qV^{*}$ and $\alpha\in \Lambda^pV^{*}$ then one has
\[\str(\beta\wedge\alpha)=(-1)^{pq}\str(\alpha\wedge\beta)=(-1)^{pq}\iota_{\alpha\wedge\beta}e=
(-1)^{pq}\iota_{\beta}\circ \iota_{\alpha}e=(-1)^{pq}\iota_{\beta}(\str \alpha).\]
After applying the Hodge star operator to both sides, we get the following useful identity
\begin{equation}\label{identity1}
\beta\wedge \str \alpha=(-1)^{q(p+1)}\str \iota_{\beta}\alpha.
\end{equation}
The vector $e\otimes e$ is a unit vector in  ${\mathcal D} V^{*}$ and doesn't depend on the chosen orientation on $V$, in fact, one can show without difficulties that $e\otimes e=\frac{g^n}{n!}$, where $g^n$ denotes the n-th exterior power of $g$ (seen as a $(1,1)$ double form in ${\mathcal D} V^{*}$.\\
In a similar way, we define the (double) Hodge star operator acting on double forms as the isomorphism $\str: {\mathcal D} V^{*}\to {\mathcal D} V^{*}$ defined by
\[ \str \omega =\iota_{\omega}(e\otimes e)=\frac{1}{n!}\iota_{\omega}g^n.\]
For a simple double form $\omega=\omega_1\otimes \omega_2$ one has 
\[ \str(\omega_1\otimes \omega_2)=\iota_{\omega_1\otimes \omega_2}e\otimes e=\iota_{\omega_1}e\otimes \iota_{\omega_2}e=\str \omega_1\otimes \str \omega_2.\]
For a $(p,q)$ double form $\theta$ and an $(r,s)$ double form $\omega$, a direct application of identity (\ref{identity1}) shows that
\begin{equation}\label{firsteq} \theta. (\str \omega)=\str (\iota_{\theta}\omega) (-1)^{p(r+1)+q(s+1)}.
\end{equation}
In particular, we obtain the following useful identity
\begin{equation}\label{identity2}
\str\Bigl( \theta. (\str \omega)\Bigr)= \iota_{\theta}\omega (-1)^{n(r+s-p-q)+r(p+1)+s(q+1)}.
\end{equation}
The above identity generalizes the identity (9) in \cite[Theorem 3.4]{Labbi-double-forms} and corrects a missing sign in it as follows.
For $p=q=1$ and $\theta=g$, the above formula reads  
\[\str\Bigl( g. (\str \omega)\Bigr)= \iota_{g}\omega (-1)^{n(r+s)}=\cc \omega (-1)^{n(r+s)},\]
or equivalently, $g.\omega=\str\Bigl( \cc  (\str \omega)\Bigr) (-1)^{n(r+s)}$. We recover formula (9) in \cite[Theorem 3.4]{Labbi-double-forms} with the correct sign in case $n(r+s)$ is not even.\\
Next, one can use the identity (\ref{firsteq}) to recover the inner product of two $(p,q)$ double forms from the Hodge star operator as follows
\begin{equation}\label{inner-hodge}
\langle \omega_1,\omega_2\rangle=\iota_{\omega_1}\omega_2=\str\bigl(\str \iota_{\omega_1}\omega_2\bigr)=\str\bigl(\omega_1.(\str\omega_2)\bigr).
\end{equation}

\subsection{Composition product and transpose of double forms}\label{transpose}
For a double form $\omega \in  {\mathcal D}^{p,q}V^*$, we denote by $\omega^t\in  {\mathcal D}^{q,p}V^*$ the transpose of $\omega$ as a bilinear form that is 
\begin{equation*}
\omega^t(u_1,u_2)=\omega(u_2,u_1).
\end{equation*}
Alternatively, for a simple double form $\omega=\theta_1\otimes \theta_2$, we set
\begin{equation*}
\omega^t=(\theta_1\otimes \theta_2)^t=\theta_2\otimes \theta_1,
\end{equation*}
and then we extend its definition by assuming the linearity of the transpose.\\
A double form $\omega$ shall be called \emph{ a symmetric double form}  if  $\omega^t=\omega$.\\
The \emph{composition product} of two simple double forms 
 $\omega_1=\theta_1\otimes \theta_2\in { \mathcal D}^{p,q}V^*$ and
    $\omega_2=\theta_3\otimes \theta_4\in  {\mathcal D}^{r,s}V^*$ is defined by
    \begin{equation*}
\omega_1\circ\omega_2=(\theta_1\otimes \theta_2)\circ (\theta_3\otimes \theta_4)=\langle \theta_1,\theta_4\rangle \theta_3\otimes \theta_2\in  {\mathcal D}^{r,q}.
\end{equation*}
It is clear that $\omega_1\circ\omega_2=0$ unless $p=s$. Then one can extends the definition to all double forms by assuming bilinearity. \\
For arbitrary double forms  $\omega_1$, $\omega_2$, the following properties hold, see \cite[Proposition 3.1]{Labbi-Bel}
\begin{itemize}
\item[a)]  $(\omega_1\circ \omega_2)^t=\omega_2^t\circ \omega_1^t$ and  $(\omega_1 \omega_2)^t=\omega_1^t\omega_2^t$.
\item[b)] If $\omega_3$ is a third double form then $\langle \omega_1\circ\omega_2,\omega_3\rangle=\langle \omega_2,\omega_1^t\circ\omega_3\rangle=\langle \omega_1,\omega_3\circ \omega_2^t\rangle.$
\end{itemize}

\subsection{The orthogonal decomposition of double forms and the first Bianchi sum}\label{decomp-Bianchi}\label{ortho-decomp}
Given a $(p,p)$ double form on an Euclidean space  $(V,g)$, there are \emph{unique trace free} $(i,i)$ double forms $\omega_i$ for $i=0,1,...,p$ 
such that  \cite{Kulkarni, Labbi-double-forms}
\begin{equation}\label{orth-decomp}
\omega=\sum_{i=0}^pg^{p-i}\omega_i.
\end{equation}
In particular, one has $\omega=0\iff \omega_i=0$ for all $i=0,1,...,p.$\\
For $n\geq 2p$, the components $\omega_i$ are uniquely determined from the different contractions of $\omega$. 
An explicit formula is given in Theorem 3.7 of \cite{Labbi-double-forms}. \\
For $n<2p$, one has $\omega_k=0$ for all $k$ such that $n-p<k\leq p$, see  \cite[Proposition 2.1]{Labbi-Balkan}. Consequently, $\omega=g^{2p-n}\bar{\omega}$ for some $(n-p,n-p)$ double form $\bar{\omega}$ and one can get the remaining components $\omega_i$ of $\omega$ from those of $\bar{\omega}$ as $n\geq 2(n-p)$.\\
Next, we define  the first Bianchi sum $\mathfrak{S}$ of double forms and its adjoint $\widetilde{\mathfrak{S}}$ as follows (see \cite{Labbi-JAUMS}). Let $(e_i)$ be an orthonormal basis of $(V,g)$ and $(e^i)$ be the dual basis. For a simple $(p,q)$ double form $\omega=\omega_1\otimes \omega_2$, we set
\begin{equation}
\mathfrak{S}\omega=\sum_{i=1}^ne^i\wedge \omega_1\otimes \iota_{e_i}\omega_2,\,\,\,\,\,\, \widetilde{\mathfrak{S}}\omega=\sum_{i=1}^n \iota_{e_i}\omega_1\otimes e^i\wedge\omega_2,
\end{equation}
we extend the definition by linearity to non-simple double forms. The following properties can be easily checked  using the above definition
\begin{align}
\widetilde{\mathfrak{S}}\omega=&\bigl(\mathfrak{S}\omega^t\bigr)^t & \mathfrak{S}\omega=&\bigl(\widetilde{\mathfrak{S}}\omega^t\bigr)^t &\\
\mathfrak{S}(g^r\omega)=&g^r \mathfrak{S}(\omega) & \widetilde{\mathfrak{S}}(g^r\omega)=&g^r \widetilde{\mathfrak{S}}(\omega) &\\
\mathfrak{S}(\cc^r\omega)=&\cc^r \mathfrak{S}(\omega) & \widetilde{\mathfrak{S}}(\cc^r\omega)=&\cc^r \widetilde{\mathfrak{S}}(\omega). &
\end{align}
Let now $\omega=\sum_{i=0}^pg^{p-i}\omega_i$ be the decomposition of $\omega$ as above. The above formulas show that $${\mathfrak{S}}\omega=\sum_{i=0}^pg^{p-i}{\mathfrak{S}}\omega_i,$$
 since $\cc {\mathfrak{S}}\omega_i={\mathfrak{S}}\cc\omega_i=0$, we conclude that the above sum is precisely the (unique) orthogonal decomposition of 
${\mathfrak{S}}$. In particular, it follows that a $(p,p)$ double form $\omega$ satisfies the first Bianchi identity if and only if all its components $\omega_i$  satisfy the first Bianchi identity for $i=0,1,...,p$.
\subsection{Other identities and properties for double forms}
Let $(V,g)$ be an Euxlidean space.
Let $h$ and $k$ be two $(1,1)$ double forms on $V$ and $p\geq 2$. The following identity follows from Greub-Vanstone basic identity, see Theorem 6.3 in \cite{Labbi-Bel}.
\begin{equation}\label{Greub-Vanstone}
\frac{g^{p-1}h}{(p-1)!}\circ \frac{g^{p-1}k}{(p-1)!}=\frac{g^{p-1}(h\circ k)}{(p-1)!}+\frac{g^{p-2}hk}{(p-2)!}
\end{equation}
In order to formulate the next identities, we first need to fix some notations.
Let $\rho$ be the following inclusion linear map
\begin{equation}
\begin{split}
\rho: \wedge^2V^*&\hookrightarrow V^*\otimes V^*\\
e^i\wedge e^j&\to \rho(e^i\wedge e^j)=e^i\otimes  e^j-e^j \otimes e^i.
\end{split}
\end{equation}
\begin{proposition}
Let now $R$ be a $(2,2)$ symmetric double  form on $V$ satisfying the first Bianchi identity and let $(E_{\alpha})$ be an orthonormal basis for $\wedge^2V^*$. The the following identities hold
Let $(e_i)$ be a basis for $V$, $(e^i)$ be the dual basis. Then the following identities hold
\begin{equation}\label{2-identiries}
\begin{split}
\sum_{\alpha}\rho\bigl(R(E_{\alpha})\bigr)\circ \rho(E_{\alpha})&=-\cc  R,\\
\sum_{\alpha}\rho\bigl(R(E_{\alpha})\bigr). \rho(E_{\alpha})&= 2R,
\end{split}
\end{equation}
where $\cc R$ denotes the Ricci contraction of $R$ and the dot (resp. the circle) denotes the exterior (resp. composition) product of double forms.
\end{proposition}

\begin{proof} The above two sums do not depend on the choice of the orthonormal basis of $\wedge^2V^*$ as they are traces of bilinear forms. Then let  $(e_i)$ be an orthonormal basis for $V$ and let $(e^i)$ be the dual basis. Then one has
\begin{equation*}
\begin{split}
\sum_{i<j} \rho\bigl(R(e^i\wedge e^j)\bigr)\circ \rho(e^i\wedge e^j) &=\frac{1}{2}\sum_{i,j} \rho\bigl(R(e^i\wedge e^j)\bigr)\circ \rho(e^i\wedge e^j) \\
&=\frac{1}{2}\sum_{i,j,k,l}R_{ijkl}(e^k\otimes e^l)\circ (e^i\otimes e^j-e^j\otimes e^i)\\
&=\frac{1}{2}\sum_{i,j,k,l}R_{ijkl}\langle e^k,e^j\rangle (e^i\otimes e^l)-\frac{1}{2}\sum_{i,j,k,l}R_{ijkl}\langle e^k,e^i\rangle (e^j\otimes e^l)\\
&=\frac{1}{2}\sum_{i,j,l}R_{ijjl} (e^i\otimes e^l)-\frac{1}{2}\sum_{i,j,l}R_{ijil} (e^j\otimes e^l)\\
&=-\frac{1}{2}\sum_{i,l}-{\mathbf c} R(e_i,e_l) (e^i\otimes e^l)-\frac{1}{2}\sum_{j,l}-{\mathbf c} R(e_j,e_l) (e^j\otimes e^l)
=-{\mathbf c} R
\end{split}
\end{equation*}
For the proof of the second identity, we shall use the first Bianchi identity as follows.

\begin{equation*}
\begin{split}
\sum_{i<j} \rho\bigl(R(e^i\wedge e^j)\bigr).&\rho(e^i\wedge e^j) =\frac{1}{2}\sum_{i,j} \rho\bigl(R(e^i\wedge e^j)\bigr). \rho(e^i\wedge e^j)\\ 
&=\frac{1}{2}\sum_{i,j,k,l}R_{ijkl}(e^k\otimes e^l). (e^i\otimes e^j-e^j\otimes e^i)\\
&=\sum_{i,j,k,l}R_{ijkl} (e^k\otimes e^l). (e^i\otimes e^j)
=\sum_{i,j,k,l}R_{ijkl} (e^k\wedge e^i)\otimes (e^l\otimes e^j)\\
&=-\sum_{i,j,k,l}R_{jkil} (e^k\wedge e^i)\otimes (e^l\otimes e^j)  -\sum_{i,j,k,l}R_{kijl} (e^k\wedge e^i)\otimes (e^l\otimes e^j)\\
&=-\sum_{i,j,k,l}R_{iljk} (e^k\wedge e^i)\otimes (e^l\otimes e^j)  +4R
=\sum_{i,j,k,l}R_{ilkj} (e^k\wedge e^i)\otimes (e^l\otimes e^j)  +4R\\
&=\sum_{i,j,k,l}R_{ijkl} (e^k\wedge e^i)\otimes (e^j\otimes e^l)  +4R
=-\sum_{i,j,k,l}R_{ijkl} (e^k\wedge e^i)\otimes (e^l\otimes e^j)  +4R\\
\end{split}
\end{equation*}
\end{proof}

\begin{proposition}[Proposition 2.3 in \cite{Labbi-double-forms}]\label{1-1}
Let $\omega$ be a $(p,q)$ double form, then
\begin{itemize}
\item[(a)]If $p+q+k<n+1$, then $g^k.\omega=0\implies \omega=0$.
\item[(b)]If $p+q+k\geq n+1$ and $p+q\leq n$, then $g^k.\omega=0\implies {\mathbf c}^r\omega=0$ for $r=p+q+k-n.$
\end{itemize}
\end{proposition}
For the seek of completeness, we provide a simple proof of the above proposition as follows. 
\begin{proof}
Lemma 2.2 in \cite{Labbi-double-forms} show that
\[{\mathbf c}^k(\frac{g^k}{k!}.\omega)=\frac{g^k}{k!}.{\mathbf c}^k \omega+\sum_{r=1}^k \binom{k}{r}\Bigl(\prod_{i=0}^{r-1}(n-p-q-i)\Bigr) \frac{g^{k-r}}{(k-r)!}.{\mathbf c}^{k-r} \omega. \]
Taking the inner of the above expression with $\omega$ and using the fact that the exterior multiplication by $g^k$ is the adjoint of the contraction map ${\mathbf c}^k$, we immediately see that
\[\frac{1}{k!}||g^k\omega||^2=\frac{1}{k!}||{\rm c}^k\omega||^2+\sum_{r=1}^k \binom{k}{r}\Bigl(\prod_{i=0}^{r-1}(n-p-q-i)\Bigr)
\frac{1}{(k-r)!}||{\mathbf c}^{k-r}\omega||^2. \]
If $n-p-q+1>k\geq 1$, then all the terms in the above sum are nonnegative and in particular one has
\[\frac{1}{k!}||g^k\omega||^2\geq \Bigl(\prod_{i=0}^{k-1}(n-p-q-i)\Bigr)||\omega||^2. \]
This completes the proof of the first part of the proposition.\\
We note that if $n=p+q$, then the product in the above sum is zero and therefore we get $||g^k\omega||^2=||{\mathbf c}^k\omega||^2$ and part b) follows in this case. Next, we assume that $n>p+q$. We remark that the product $\prod_{i=0}^{r-1}(n-p-q-i)$ is zero if $0\leq n-p-q\leq r-1$. Hence we get
\begin{equation*}
\begin{split}
\frac{1}{k!}||g^k\omega||^2&=\frac{1}{k!}||{\mathbf c}^k\omega||^2+\sum_{r=1}^{n-p-q} \binom{k}{r}\Bigl(\prod_{i=0}^{r-1}(n-p-q-i)\Bigr)
\frac{1}{(k-r)!}||{\mathbf c}^{k-r}\omega||^2\\
&\geq  \binom{k}{n-p-q}\Bigl(\prod_{i=0}^{n-p-q-1}(n-p-q-i)\Bigr)
\frac{1}{(k-n+p+q)!}||{\mathbf c}^{k-n+p+q}\omega||^2.
\end{split}
\end{equation*}
Therefore, if $p+q+k\geq n+1$ then the condition $g^k\omega=0$ implies ${\mathbf c}^{p+q+k-n}\omega=0$. This completes the proof of the proposition.
\end{proof}
\section{Extensions of endomorphisms and Ricci type identities}\label{extensions-section}
Let $(V,g)$ be an Euclidean space of finite dimension $n$. Let $h:V\to V$ be a vector space endomorphism and let $\wedge h:\wedge V\to \wedge V$ denotes its extension by derivations to the exterior algebra $\wedge V$. We denote as well by $h$  the bilinear form associated to the operator $h$ via the metric $g$: $h(u,v):=g(h(u),v)$. It turns out that the bilinear form associated with the above operator $\wedge h$ via the standard inner product on $\wedge V$ is given by the exterior product of double forms as follows 
\begin{equation} (\wedge h) (\theta_1,\theta_2):=\langle \wedge h(\theta_1),\theta_2\rangle =\frac{g^{p-1}h}{(p-1)!}(\theta_1,\theta_2),\,\,\, {\rm}\,\, \theta_1,\theta_2\in \wedge^pV.
\end{equation}
Let $h^{\flat}:V^*\to V^*$ be the lifted operator, defined via the following commutative diagram (this not the dual of $h$ but it is instead the dual of its adjoint operator),
\begin{center}
\begin{tikzpicture}[scale=2]
\node (A) at (0,1) {$V$};
\node (B) at (1,1) {$V$};
\node (C) at (0,0) {$V^*$};
\node (D) at (1,0) {$V^*$};
\path[->,font=\scriptsize,>=angle 90]
(A) edge node[above]{$h$} (B)
(A) edge node[right]{$\flat$} (C)
(B) edge node[right]{$\flat$} (D)
(C) edge node[above]{$h^{\flat}$} (D);
\end{tikzpicture}
\end{center}
In particular, one has
\begin{equation}
h^\flat(\alpha)=\alpha\circ \bar{h},\,\,\, {\rm for}\,\, \alpha\in V^*,
\end{equation}
where $\bar{h}$ denotes the adjoint of $h$, and $\flat, \sharp=\flat^{-1}$ are the musical isomorphisms.\\
The spaces $(\wedge V)^*$ and $\wedge V^*$ are naturally isomorphic, let then $(\wedge h)^\flat:\wedge V^*\to \wedge V^*$ be the dual operator to $\wedge h$ as above.  Then one has
\[(\wedge h)^\flat(\theta)=\theta \circ (\overline{\wedge h})=\theta\circ \wedge \bar{h},\]
where the last equality can be for instance verified at the level of the corresponding bilinear forms as we have 
 $(g^{p-1}h)^t=(g^{p-1})^t h^t=g^{p-1}h^t$, see for instance \cite{Labbi-Bel}.
\begin{proposition}\label{ext-forms}
For any $\theta \in \wedge^pV^*$ one has
\begin{equation}
(\wedge h)^\flat(\theta)=(\theta\otimes 1) \circ \frac{g^{p-1}h^t}{(p-1)!},
\end{equation}
where $\circ$ denotes the composition of double forms and the $(p,0)$ double form on the right hand side of the equation is identified to a $p$-form.
\end{proposition}
\begin{proof}
Let  $(e_i)$ be an orthonormal basis of $V$ and $(e^i)$ the dual basis. Let capital $I,J$ denote multi-indices for corresponding bases in the exterior algebras, then
\begin{equation*}
\begin{split}
\theta\circ \wedge h^t=& \sum_I \bigl(\theta\circ \wedge h^t(e_I)\bigr)e^I\\
=& \sum_I \theta(\wedge h^t(e_I))e^I=\sum_{I,J} \wedge h^t(e_I,e_J))\theta(e_J)e^I\\
=&\sum_{I,J} \wedge h^t(e_I,e_J))(\theta\otimes 1)\circ (e^I\otimes e^J)=(\theta\otimes 1)\circ \wedge h^t.
\end{split}
\end{equation*}
Note that in the above proof, the first $\circ$ is the usual composition of maps while the last ones denote the composition product of double forms. It was for the seek of simplifying notations that we denoted in this paper operators and their associated bilinear forms by the same symbols!
\end{proof}

Net, we extend $(\wedge h)^\flat$ to operate on double forms in four ways. For a simple double form $\theta^1\otimes\theta^2$ we set and then we extend by linearity
 \begin{equation}
\begin{split}
(\wedge h)_\ell^\flat(\theta^1\otimes\theta^2)&:=(\wedge h)^\flat(\theta^1)\otimes \theta^2,\\
(\wedge h)_r^\flat(\theta^1\otimes\theta^2) &:=\theta^1\otimes (\wedge h)^\flat(\theta^2),\\
(\wedge h)_{d}^\flat &:=(\wedge h)_\ell^\flat+(\wedge h)_r^\flat\\
(\wedge h)_{{\scriptsize \Delta}}^\flat &:=(\wedge h)^\flat(\theta^1)\otimes (\wedge h)^\flat(\theta^2).
\end{split}
\end{equation}

\begin{proposition}\label{ext-formulas}
 Let $\omega$ be a $(p,q)$ double form, then the following properties hold
 \begin{equation}
\begin{split}
(\wedge h)_\ell^\flat(\omega)&=\omega \circ \frac{g^{p-1}h^t}{(p-1)!} \\
(\wedge h)_r^\flat(\omega ) &=\frac{g^{q-1}h}{(q-1)!}\circ \omega, \\
(\wedge h)_{d}^\flat &=\omega \circ \frac{g^{p-1}h^t}{(p-1)!}+\frac{g^{q-1}h}{(q-1)!}\circ \omega.\\
(\wedge h)_{{\scriptsize \Delta}}^\flat &=\frac{g^{q-1}h}{(q-1)!}\circ \omega \circ \frac{g^{p-1}h^t}{(p-1)!}.\\
\end{split}
\end{equation}
\end{proposition}

\begin{proof}
The proof is similar to the proof of the above proposition, using the same notations one has
\begin{equation*}
\begin{split}
(\wedge h)_\ell^\flat(\theta^1\otimes\theta^2)&=(\wedge h)^\flat(\theta^1)\otimes \theta^2=(\theta^1\circ (\wedge h)^t)\otimes \theta^2,\\
=& \sum_I \theta^1(\wedge h^t(e_I))e^I\otimes \theta^2=\sum_{I,J} \wedge h^t(e_I,e_J))\theta^1(e_J)e^I\otimes \theta^2,\\
=&\sum_{I,J} \wedge h^t(e_I,e_J))(\theta^1\otimes \theta^2)\circ (e^I\otimes e^J)=(\theta^1\otimes \theta^2)\circ \wedge h^t.
\end{split}
\end{equation*}
The second assertion can be proved using the first one in the following way
\[ (\wedge h)_r^\flat(\omega)=\bigl((\wedge h)_\ell^\flat (\omega^t)\bigr)^t=\bigl(\omega^t \circ \frac{g^{q-1}h^t}{(q-1)!}\bigr)^t=\frac{g^{q-1}h}{(q-1)!}\circ \omega.\]
The third  and fourth parts follow directly from parts one and two of the same proposition.
\end{proof}
\begin{proposition}\label{derivations}
The three operators  $(\wedge h)_\ell^\flat, (\wedge h)_r^\flat$ and $(\wedge h)_{d}^\flat$ act by derivations on the exterior algebra of double forms.
\end{proposition}

\begin{proof}
Let $\omega=\theta^1\otimes\theta^2$ and $\theta=\theta^3\otimes\theta^4$ be simple double forms, then

\begin{equation*}
\begin{split}
 (\wedge h)_r^\flat(\omega.\theta)&= (\wedge h)_r^\flat\left((\theta^1\otimes\theta^2).(\theta^3\otimes\theta^4)\right)\\
&= (\wedge h)_r^\flat((\theta^1\wedge\theta^3)\otimes(\theta^2\wedge\theta^4))\\
&=(\theta^1\wedge\theta^3)\otimes (\wedge h)^\flat(\theta^2\wedge\theta^4)\\
&=(\theta^1\wedge\theta^3)\otimes [(\wedge h)^\flat(\theta^2)\wedge\theta^4+\theta^2\wedge (\wedge h)^\flat(\theta^4)]\\
&=[(\theta^1\wedge\theta^3)\otimes (\wedge h)^\flat(\theta^2)\wedge\theta^4]+[(\theta^1\wedge\theta^3)\otimes\theta^2\wedge (\wedge h)^\flat(\theta^4)]\\
&=(\theta^1\otimes (\wedge h)^\flat(\theta^2)).(\theta^3\otimes\theta^4)+(\theta^1\otimes\theta^2).(\theta^3\otimes (\wedge h)^\flat(\theta^4))\\
&=(  (\wedge h)_r^\flat(\omega)).\theta+\omega  .(\wedge h)_r^\flat(\theta).
\end{split}
\end{equation*}
The proofs for  $(\wedge h)_\ell^\flat$ and $(\wedge h)_{d}^\flat$ are similar.
\end{proof}
\subsection{Applications}
\subsubsection{Curvature transformation and its extensions by derivations} Let  $R$ be a $(2,2)$  double form on $V$.
Let $(e_i)$ be an orthonormal basis of $V$. In what follows we shall denote by $R_{ij}$ the curvature transformation $R_{e_ie_j}$. It is the $(1,1)$ double form given by $R(e_i,e_j, .,.)$,  we shall use the same notation to denote the corresponding operator $V\to V$.\\
We will continue to use the same symbol $R_{ij}$ to denote the extension by derivations of the operator $R_{ij}$ to the exterior algebra $\wedge V^*$  (that is the $(\wedge R_{ij})^\flat$ of the above section)  and to the exterior algebra of double forms $\wedge V^*\otimes \wedge V^*$ ( that is the $(\wedge R_{ij})_{rl}^\flat$ of the above section).\\
As a direct consequence of Proposition \ref{ext-formulas}, we obtain the following Ricci type identities 

\begin{proposition}\label{rij-derivations}
Let $\omega$ be a $(p,q)$ double form, then
\[ R_{ij}\omega=\frac{g^{q-1}R_{ij}}{(q-1)!}\circ \omega-\omega\circ \frac{g^{p-1}R_{ij}}{(p-1)!},\]
with the convention that $g^{-1}=0$.
\end{proposition}
\begin{proof}
Recall that $(R_{ij})^t=-R_{ij}$ and use Proposition \ref{ext-formulas}.
\end{proof}
\noindent
\begin{remark} Note that for $\omega=h$ a $(1,1)$ double form the previous formula read $R_{ij}(h)=R_{ij}\circ h-h\circ R_{ij}$. In particular, we recover the following  Ricci identity
$$R_{ij}(h)(z,u)=-h(R_{ij}z,u)-h^t(R_{ij}u,z).$$
\end{remark}
\subsubsection{The transformation $(\rho_{ij})$  and its extensions}
Let $(e_i)$ be a basis for $V$ and $(e^i)$ be the dual basis. Let $\rho$ be the following inclusion map
\begin{equation}
\begin{split}
\rho: \wedge^2V^*&\hookrightarrow V^*\otimes V^*\\
e^i\wedge e^j&\to \rho(e^i\wedge e^j)=e^i\otimes  e^j-e^j \otimes e^i.
\end{split}
\end{equation}
To simplify notations, we shall denote by $\rho_{ij}$ the above $(1,1)$ double form $\rho(e^i\wedge e^j)$, we denote as well by $\rho_{ij}$ the associated operator $V\to V.$  Namely, $\rho_{ij}(v)=e^i(v)e_j-e^j(v)e_i.$ 
Let $\rho_{ij}$ denote also the extension by derivations of the operator $\rho_{ij}$ to the exterior algebra $\wedge V^*$, that is the $(\wedge \rho_{ij})^\flat$ of the above section.  We denote by $(\rho_{ij})_\ell$  the left extension of $\rho_{ij}$ to the exterior algebra of double forms $\wedge V^*\otimes \wedge V^*$ ( that is the $(\wedge \rho_{ij})_{rl}^\flat$ of the above section).\\

\begin{proposition}\label{4parts-proposition}
In what follows, in the case $p=0$, we use the convention $g^{-1}=0$.
\begin{enumerate}
\item 
For all $\theta\in \wedge^pV^*$ one has 
\[\rho_{ij}(\theta)=-(\theta\otimes 1) \circ \frac{g^{p-1}\rho_{ij}}{(p-1)!},\]
where the right hand side $(p,0)$ double form is seen as a $p$ form. 
\item For all $\omega\in \wedge^pV^*\otimes \wedge^qV^*$ one has 
\[ (\rho_{ij})_\ell(\omega)=-\omega \circ \frac{g^{p-1}\rho_{ij}}{(p-1)!}\,\,\, \,\,\,\, \omega\in \wedge^pV^*\otimes \wedge^qV^*.\]
\item For any double forms $\omega_1, \omega_2$ we have
\[(\rho_{ij})_\ell(\omega_1\circ \omega_2)=\omega_1\circ (\rho_{ij})_\ell(\omega_2).\]
\item For any double forms $\omega_1, \omega_2$ we have
\[(\rho_{ij})_\ell(\omega_1.\omega_2)=(\rho_{ij})_\ell(\omega_1).\omega_2+\omega_1. (\rho_{ij})_\ell(\omega_2).\]

\end{enumerate}
\end{proposition}

\begin{proof}
Parts (1) and (2) follow directly from Propositions \ref{ext-forms} and  \ref{ext-formulas} and the skew-symmetry of the $\rho_{ij}$. Part (3) is an immediate consequence of Part (2) and part (4) results from  Proposition \ref{derivations}.
\end{proof}

The following proposition shows that the operator $\rho_{ij}$ on double forms coincides with an algebraic operator that naturally appears in Weitzenb\"ock formulas for double forms.
\begin{proposition}
For any $(p,q)$ doule form on $V$ one has 
\[\rho_{ij}(\omega)=\bigl[i_{e_j\otimes 1}\mu_{e^i\otimes 1}-i_{e_i\otimes 1}\mu_{e^j\otimes 1}\bigr](\omega),\]
where for a given double form $\theta$, $\mu_\theta$ denotes the left multiplication by $\theta$ in the exterior algebra of double forms.
\end{proposition}
\begin{proof}
It is straightforward to see that the right hand side operator, like  the operator $\rho_{ij}$, acts on left factors of double forms. It's also immediate that the same operator acts by derivations on forms ($(p,0)$ forms). Consequently, it suffices to check that the above operators coincide on $V^*$. For, let $\alpha\in V^*$, then
\[\bigl[\iota_{e_i\otimes 1}\mu_{e^j\otimes 1}-i_{e_j\otimes 1}\mu_{e^i\otimes 1}\bigr](\alpha)=i_{e_i}(e^j\wedge \alpha)-i_{e_j}(e^i\wedge \alpha)=\alpha(e_j)e^i-\alpha(e_i)e^j.\]
On the other hand, the above proposition shows that
\[\rho_{ij}(\alpha)=-(\alpha\otimes 1)\circ \rho_{ij}=-(\alpha\otimes 1)\circ (e^i\otimes e^j-e^j\otimes e^i)=-\alpha(e_j)e^i+\alpha(e_i)e^j.\]
This completes the proof.
\end{proof}
\subsection{Left and right multiplications maps in the composition algebra and their extensions}
Let $h$ be a $(1,1)$ double form, we define the endomorphisms $\lambda_h, \rho_h$ of ${\mathcal D}^{1,1}V^*$  by
\[\lambda_h(k)=h\circ k,\,\,\,\, {\rm and}\,\, \rho_h(k)=k\circ h.\]
We continue to denote by $\lambda_h, \rho_h$ their extensions by derivations to the (diagonal) exterior sub-algebra of double forms ${\bigoplus_{p\geq 0}\mathcal D}^{p,p}V^*$. On the other side, the extensions $(\wedge h)_\ell^\flat$ and $ (\wedge h)_r^\flat$ as in section \ref{extensions-section} also operate by derivations on the exterior algebra by Proposition \ref{derivations}. Furthermore,  their restrictions to 
${\mathcal D}^{1,1}V^*$ coincide respectively  with $ \rho_{h^t}$ and $\lambda_h$   of ${\mathcal D}^{1,1}V^*$ respectively. We have therefore proved that for a $(p,p)$ double form one has
\begin{equation}
\lambda_h(\omega)=\frac{g^{p-1}h}{(p-1)!}\circ \omega,\,\, \rho_h(\omega)=\omega\circ \frac{g^{p-1}h}{(p-1)!}.
\end{equation}
In particular, if $R$ is a $(2,2)$ double form, one gets
\begin{equation}\label{hR} \lambda_h(R)(x,y,z,u)=\bigl(gh\circ R\bigr)(x,y,z,u)=h(R(x,y)z,u)-h(R(x,y)u,z).\end{equation}
\begin{remark} We defined in \cite{Labbi-variational} the operator $F_h$ operating on the exterior algebra of double forms by derivations and its restriction to $(1,1)$ double forms is $\lambda_h+\rho_h$. This operator played a key role in the proof of the main theorem in \cite{Labbi-variational}. It results from the above discussion, that for a $(p,p)$ double form one has
\begin{equation}
F_h(\omega)=\lambda_h(\omega)+\rho_h(\omega)=\frac{g^{p-1}h}{(p-1)!}\circ \omega +\omega\circ \frac{g^{p-1}h}{(p-1)!}.
\end{equation}

\end{remark}
\section{Clifford multiplication on forms and the sharp product of double forms}
\subsection{Lie algebra of forms and the adjoint representation}
 Let $e\in V^*$ and $\alpha\in \wedge^pV^*$. The Clifford product, denoted by a dot ".",  of $e$ and $\alpha$ is given by
\begin{equation}
\begin{split}
e.\alpha:=&e\wedge \alpha-i_{e}\alpha,\\
\alpha.e:=& (-1)^p\bigl( e\wedge \alpha+i_{e}\alpha\bigr).
\end{split}
\end{equation}
We can extend the above product to any two forms after  assuming that it is associative and bilinear. With this product, $\wedge V^*$ becomes an algebra isomorphic to the Clifford algebra ${\mathrm{Cl}}(V,g)$.\\
The Clifford algebra $\wedge V^*$ which is associative can be turned into a Lie algebra via the commutator bracket  $[\alpha, \beta]=\alpha.\beta-\beta.\alpha$, which we shall call the Clifford Lie bracket. 
For $\alpha\in \wedge V^*$, the adjoint representation $\mathrm{ad}_\alpha:\wedge V^*\to\wedge V^*$ is defined by
\[\mathrm{ad}_\alpha(\omega)=[\alpha,\omega].\]
\begin{proposition}\label{ad-rho}
Let $\alpha\in \wedge^2 V^*$ and $\omega\in \wedge V^*$, then one has
\begin{equation}
\mathrm{ad}_\alpha(\omega)=2\rho(\alpha)\omega=2\sum_{i=1}^ni_{e_i}\alpha\wedge i_{e_i}\omega,
\end{equation}
where $\rho(\alpha)$ is the transformation defined in the above section, $\iota$ denotes the interior product and $(e_i)$ is any orthonormal basis of $V$.
\end{proposition}

\begin{proof}
Let $\omega\in \wedge^pV^*$, then straightforward computations show that
\begin{equation*}
\begin{split}
(e^i\wedge e^j).\omega=&e^i.e^j.\omega+g(e_i,e_j)\omega\\
=& e^i.(e^j\wedge\omega-i_{e_j}\omega)+g(e_i,e_j)\omega\\
=& e^i.(e^j\wedge\omega)-e^i. i_{e_j}\omega+g(e_i,e_j)\omega\\
=& e^i\wedge e^j\wedge\omega-i_{e_i}(e^j\wedge \omega)-e^i\wedge i_{e_j}\omega +i_{e_i}i_{e_j}\omega+g(e_i,e_j)\omega\\
=& e^i\wedge e^j\wedge\omega+e^j\wedge i_{e_i}\omega-e^i\wedge i_{e_j}\omega+i_{e_j\wedge e_i}\omega 
+g(e_i,e_j)\omega.
\end{split}
\end{equation*}
On the other hand, one has
\begin{equation*}
\begin{split}
\omega.(e^i\wedge e^j)=&\omega.e^i.e^j+g(e_i,e_j)\omega\\
=& (-1)^p\bigl( e^i\wedge\omega+i_{e_i}\omega\bigr).e^j+g(e_i,e_j)\omega\\
=& (-1)^p( e^i\wedge\omega).e^j+(-1)^p(i_{e_i}\omega).e^j+g(e_i,e_j)\omega\\
=& (-1)^p(-1)^{p+1}\bigl[ e^j\wedge e^i\wedge\omega+i_{e_j}(e^i\wedge \omega)\bigr]\\
&\,\,\,\,\,\,\,\,\,\,\,\,\,\,\,\,\,\, + (-1)^p(-1)^{p-1}\bigl[e^j\wedge i_{e_i}\omega+i_{e_j}i_{e_i}\omega\bigr]+g(e_i,e_j)\omega\\
=& e^i\wedge e^j\wedge\omega+e^i\wedge i_{e_j}\omega-e^j\wedge i_{e_i}\omega-i_{e_i\wedge e_j}\omega 
+g(e_i,e_j)\omega.
\end{split}
\end{equation*}
Consequently, one has
\begin{equation*}
\begin{split}
\mathrm{ad}_{e^i\wedge e^j}\omega=&(e^i\wedge e^j).\omega-\omega.(e^i\wedge e^j)\\
=&2e^j\wedge i_{e_i}\omega-2e^i\wedge i_{e_j}\omega\\
=&-2i_{e_i}(e^j\wedge \omega)+2i_{e_j}(e^i\wedge \omega)=2\rho_{ij}(\omega).\\
\end{split}
\end{equation*}
This completes the proof of the first equality. The second equality follows by remarking that
\begin{equation*}
\sum_{k=1}^n i_{e_k}(e^i\wedge e^j)\wedge i_{e_k}\omega=e^j\wedge i_{e_i}\omega-e^i\wedge i_{e_j}\omega.
\end{equation*}
\end{proof}

The space of $(1,1)$-double forms $ V^*\otimes V^*$ is trivially a Lie algebra under the commutator bracket associated to the composition product of double forms, which we shall denote by $[.,.]_\circ$ and refer to it as the composition Lie bracket. Precisely, $[\omega_1, \omega_2]_\circ=\omega_1\circ \omega_2-\omega_2\circ \omega_1$. Since the map $\rho:\wedge^2V^*\to V^*\otimes V^*$ is bijective onto the subspace of skew symmetric $(1,1)$ double forms, one can then pull back the bracket to $\wedge^2V^*$ as follows
\[[\alpha_1,\alpha_2]_\circ :=\frac{1}{2}\rho^{-1}\bigl([\rho(\alpha_1),\rho(\alpha_2)]_\circ\bigr),\]
and hence make $(\wedge^2V^*, [.,.]_\circ)$ a Lie algebra. \\It turns out that the linear map $\mathrm{ad}:\wedge^2V^*\to \rm{End}(\wedge V^*)$ is a representation of the Lie algebra $(\wedge^2V^*, [.,.]_\circ)$. This can be easily proved using the above proposition as follows
\begin{equation*}
[\mathrm{ad}_\alpha,\, \mathrm{ad}_\beta]_{\circ}= 4[\rho(\alpha),\, \rho(\beta)]_{\circ}=2\rho([\alpha,\beta]_\circ)=\mathrm{ad}_{[\alpha,\beta]}.\\
\end{equation*}
The next proposition shows that the above two Lie brackets on $2$-forms coincides up to a factor
\begin{proposition}
Let $[.,.]$ and $[.,.]_\circ$ denote respectively the Clifford and composition Lie brackets on $\wedge^2V^*$, then one has
\begin{equation}
[\alpha,\beta]=4[\alpha,\beta]_\circ,
\end{equation}
for all $\alpha, \beta \in \wedge^2V^*$
\end{proposition}
\begin{proof}
For arbitrary vectors $a,b\in V$ and orthonormal basis $(e_i)$  of $V$, one has using the definition of the composition product of double forms
\begin{equation*}
\begin{split}
[\rho(\alpha),\rho(\beta)]_\circ(a,b)=& \rho(\alpha)\circ\rho(\beta)(a,b)-\rho(\beta)\circ\rho(\alpha)(a,b)\\
=& \sum_{i=1}^n\Bigl( \rho(\beta)(a,e_i)\rho(\alpha)(e_i,b)-\rho(\alpha)(a,e_i)\rho(\beta)(e_i,b)\Bigr)\\
=& \sum_{i=1}^n \Bigl(\beta(a,e_i)\alpha(e_i,b)-\alpha(a,e_i)\beta(e_i,b)\Bigr)\\
=& -\sum_{i=1}^n\Bigl( i_{e_i}\beta(a)i_{e_i}\alpha(b)-i_{e_i}\alpha(a)i_{e_i}\beta(b)\Bigr)\\
=& \sum_{i=1}^n \rho(i_{e_i}\alpha\wedge i_{e_i}\beta)(a,b)=\frac{1}{2}\rho([\alpha,\beta])(a,b).
\end{split}
\end{equation*}
\end{proof}

\subsection{Sharp product of double forms}
Recall that the Clifford Lie bracket $[.,.]$ in $\wedge V^*$ is given by the Clifford product of forms as follows $[\alpha,\beta]=\alpha.\beta-\beta.\alpha.$ This makes $(\wedge V^*,[.,.])$ a Lie algebra. The sharp product $\#$ is by definition the natural product in the algebra tensor product $\wedge V^*\otimes\wedge V^*$. Precisely, we define
\begin{definition}
For two simple double forms $\alpha_1\otimes \alpha_2$ and $\beta_1\otimes \beta_2$ we define
\begin{equation}
(\alpha_1\otimes \alpha_2) \# (\beta_1\otimes \beta_2):=[\alpha_1,\beta_1]\otimes [\alpha_2,\beta_2].
\end{equation}
By assuming bilinearity, this can be extended to a product on the space of double forms.
\end{definition}
This product is clearly symmetric and generalizes the sharp product that arises in the evolution equation of the Riemann curvature tensor under the Ricci flow, see corollary \ref{sharp-use} below. This product was first defined by B\"ohm and Wilking for $(2,2)$ double forms in \cite{BW}, see also \cite{Shmidt}.\\
We remark from the definition that the sharp product results in a zero double form if one of the factors is a $(p,0)$ or a $(0,q)$ double form.
\begin{proposition}\label{sharp-forms}
\begin{enumerate}
\item Let $h$ be a $(1,1)$ double form and $\omega$ be a $(p,p)$ double form.
\begin{itemize}
\item[a)] If $p$ is even then $h\# \omega=4i_{h}\omega$ is the interior product of the $\omega$ with  $h$.
\item[b)] If $p$ is odd, then $h\# \omega=4h\omega$ is the exterior product of the two double forms.
\end{itemize}
\item If $\alpha=\alpha_1\otimes \alpha_2$ is a simple $(2,2)$ double form and $\omega$ be a $(p,q)$ double form then (with the convention that $g^{-1}=0$)
\begin{equation}\label{sharp-composition}
\alpha \#\omega =-4\frac{g^{q-1}\rho(\alpha_2)}{(q-1)!}\circ\omega\circ \frac{g^{p-1}\rho(\alpha_1)}{(p-1)!}.
\end{equation}
\item For a $(2,2)$ double form  $R$  and a $(p,q)$ double form $\omega$ one has
\begin{equation}\label{sharp-interior}
R \# \omega =4\sum_{j,k=1}^n (i_{e_j\otimes e_k}R).(i_{e_j\otimes e_k}\omega),
\end{equation}
where the dot in the last expression is the exterior product of the two double forms and $(e_i)$ is an orthonormal basis of $V$.
\end{enumerate}
\end{proposition}
\begin{proof}
For $e\in V^*$, we first remark that
\[ [e,\omega]=e.\omega-\omega.e=e\wedge \omega-i_{e}\omega-(-1)^p(e\wedge \omega+i_{e}\omega).\]
That is equal to $-2i_{e}\omega$ if $p$ is even and equal to $2e\wedge \omega$ if $p$ is odd.\\
Without loss of generality we suppose $h=e^1\otimes e^2$ and $\omega=\theta^1\otimes \theta^2$, then if $p$ is even one has
\[h\# \omega=[e^1,\theta^1]\otimes [e^2,\theta^2]=4i_{e_1}\theta^1\otimes i_{e_2}\theta^2=4i_{h}\omega.\]
The case $p$ odd can be proved in a similar way.\\
For the second part, let $\alpha=\alpha_1\otimes\alpha_2$ and without loss of generality let $\omega=\omega_1\otimes \omega_2$. Using Proposition \ref{ad-rho} and then Proposition \ref{ext-formulas} one can easily see that
\begin{equation*}
\begin{split}
\alpha\# \omega=&\mathrm{ad}_{\alpha_1}\omega_1\otimes \mathrm{ad}_{\alpha_2}\omega_2\\
=&4\rho(\alpha_1)\omega_1\otimes \rho(\alpha_2)\omega_2=(\rho(\alpha_1))_\ell (\rho(\alpha_2))_r\omega_1\otimes \omega_2\\
=& 4\frac{g^{q-1}\rho(\alpha_2)}{(q-1)!}\circ \omega \circ \Bigl(\frac{g^{p-1}\rho(\alpha_2)}{(p-1)!}\Bigr)^t=-4\frac{g^{q-1}\rho(\alpha_2)}{(q-1)!}\circ \omega \circ \frac{g^{p-1}\rho(\alpha_2)}{(p-1)!}.
\end{split}
\end{equation*}
Finally, we shall prove the last part using the second equality in Proposition \ref{ad-rho}. Without loss of generality assume $R=\alpha_1\otimes \alpha_2$ and $\omega=\omega_1\otimes \omega_2$ then
\begin{equation*}
\begin{split}
[R,\omega]=& \mathrm{ad}_{\alpha_1}\omega_1\otimes \mathrm{ad}_{\alpha_2}\omega_2\\
=&4\sum_{j,k=1}^n (i_{e_j}\alpha_1 \wedge i_{e_j}\omega_1)\otimes (i_{e_k}\alpha_2 \wedge i_{e_k}\omega_2)\\
=& 4\sum_{j,k=1}^n (i_{e_j}\alpha_1 \otimes i_{e_k}\alpha_2).( i_{e_j}\omega_1\otimes i_{e_k}\omega_2 )=4\sum_{j,k=1}^n (i_{e_j\otimes e_k}R).(i_{e_j\otimes e_k}\omega).
\end{split}
\end{equation*}
\end{proof}
\begin{corollary}\label{g-square-sharp}
Let $\omega$ be a $(p,q)$ double form and let $\mathfrak{S}$ and $\widetilde{\mathfrak{S}}$ respectively denote the first Bianchi sum and its adjoint sum. Then the following holds
\begin{equation}
\frac{1}{2}g^2\# \omega=4g.\cc {\omega}-4q\omega+4\widetilde{\mathfrak{S}}\mathfrak{S}\omega.
\end{equation}
\end{corollary}
\begin{proof}
\[
\frac{1}{2}g^2\# \omega=2\sum_{j,k=1}^n (i_{e_j\otimes e_k}g^2).(i_{e_j\otimes e_k}\omega)=2\sum_{j=1}^n (i_{e_j\otimes e_j}g^2).(i_{e_j\otimes e_j}\omega)+2\sum_{j\not=k}^n (i_{e_j\otimes e_k}g^2).(i_{e_j\otimes e_k}\omega).
\]
One can see without difficulties that 
\[  i_{e_j\otimes e_k}g^2= \left\{
\begin{array}{ll}
      -2e^k\otimes e^j &{\rm if}\,\,  j\not= k \\
      2g-2e^j\otimes e^j & {\rm if}\,\,  j=k \\
\end{array} 
\right. \]
Consequently, we have 
\[
\begin{split}
\frac{1}{2}g^2\# \omega=&2\sum_{j=1}^n\bigl[ 2g.i_{e_j\otimes e_j}\omega-2(e^j\otimes e^j). i_{e_j\otimes e_j}\omega\bigr]
+2\sum_{j\not=k}^n  -2(e^k\otimes e^j).(i_{e_j\otimes e_k}\omega)\\
=&4g.\cc \omega -4 \sum_{j,k}^n (e^k\otimes e^j).(i_{e_j\otimes e_k}\omega).
\end{split}
\]
Recall that the first Bianchi sum and its adjoint are given by (see Proposition 3.2 and Remark 3.2 in \cite{Labbi-JAUMS})
\[ \mathfrak{S}\omega=\sum_{i=1}^n (e^i\otimes 1).i_{1\otimes e_i}\omega,\,\,\, {\rm and}\,\,\, \widetilde{\mathfrak{S}}\omega=\bigl(\mathfrak{S}(\omega^t)\bigr)^t
=\sum_{i=1}^n i_{e_i\otimes 1}\bigl((1\otimes e^i)\omega\bigr).
\]
Using the properties of the transpose, see section \ref{transpose},  and formula (\ref{2identities}) for the  interior product of double forms, one can show that
\[
\begin{split}
\widetilde{\mathfrak{S}}\mathfrak{S}\omega=& \widetilde{\mathfrak{S}}\Bigl[ \sum_{i=1}^n (e^i\otimes 1).i_{1\otimes e_i}\omega\Bigr]=
\Bigl({\mathfrak{S}}\Bigl[ \sum_{i=1}^n (e^i\otimes 1).i_{1\otimes e_i}\omega\Bigr]^t\Bigr)^t\\
=& \Bigl({\mathfrak{S}}\Bigl[ \sum_{i=1}^n (1\otimes e^i).i_{e_i\otimes 1}\omega^t\Bigr]\Bigr)^t
=\Bigl( \sum_{i,j=1}^n (e^j\otimes 1). i_{1\otimes e_j}\bigl[ (1\otimes e^i).i_{e_i\otimes 1}\omega^t\bigr]\Bigr)^t\\
=& \Bigl( \sum_{i=1}^n (e^i\otimes 1).i_{e_i\otimes 1}\omega^t-\sum_{i,j=1}^n (e^j\otimes e^i). i_{e_i\otimes e_j}\omega^t\Bigr)^t\\
&=  \sum_{i=1}^n (1\otimes e^i).i_{1\otimes e_i}\omega-\sum_{i,j=1}^n (e^i\otimes e^j). i_{e_j\otimes e_i}\omega=q\omega-\sum_{i,j=1}^n (e^i\otimes e^j). i_{e_j\otimes e_i}\omega.
\end{split}
\]
\end{proof}

\section{Hodge-de Rham Laplacian on double forms and Weitzenb\"ock formula}
Let $(M,g)$ be a smooth Riemannian  $n$-manifold and, $\nabla$ is its Levi-Civita connection and $X$ be a vector field on $M$. The covariant differential of $X$ is the endomorphism  $\nabla X$ defined by
$$\nabla X(Y)=\nabla_YX.$$
We assume the Leibniz rule to define the covariant of more general tensors. In particular, for a $(p,q)$-double form $\omega$, its covariant differential $\nabla \omega$, is defined to be the $(0,p+q+1)$ tensor given by
\begin{equation*}
\begin{split}
\nabla\omega(X,Y_1,...,Y_p,Z_1,...,Z_q):=&(\nabla_X\omega)(Y_1,...,Y_p,Z_1,...,Z_q)\\
=& \nabla_X\Bigl[ \omega(Y_1,...,Y_p,Z_1,...,Z_q)\Bigr]-\omega(\nabla_xY_1,Y_2,...,Y_p,Z_1,...,Z_q) \\
&  -\omega(Y_1, \nabla_xY_2,...,Y_p,Z_1,...,Z_q)-...-\omega(Y_1, Y_2,...,Y_p,Z_1,...,\nabla_X Z_q).
\end{split}
\end{equation*}
\begin{proposition}
Let $x\in M$ and $(e_i)$ be a basis of the tangent space $T_xM$. Let $(e^i)$ denotes the dual basis, then for a double form $\omega$ the following holds at  $x$ 
\begin{equation}
\nabla\omega =\sum_{i=1}^ne^i\otimes \nabla_{e_i}\omega.
\end{equation}
\end{proposition}
\begin{proof}
Let $X,Y_1,...,Y_p,Z_1,...,Z_q$ be tangent vectors then 
\begin{equation*}
\begin{split}
\sum_{i=1}^ne^i\otimes \nabla_{e_i}&\omega(X,Y_1,...,Y_p,Z_1,...,Z_q)=\sum_{i=1}^ne^i(X) \nabla_{e_i}\omega(Y_1,...,Y_p,Z_1,...,Z_q)\\
&=\nabla_{\sum_{i=1}^ne^i(X)  e_i}\omega(Y_1,...,Y_p,Z_1,...,Z_q)=\nabla_X\omega(Y_1,...,Y_p,Z_1,...,Z_q)\\
&=\nabla\omega(X,Y_1,...,Y_p,Z_1,...,Z_q).
\end{split}
\end{equation*}
\end{proof}
The second Bianchi sum of a $(p,q)$ double form $\omega$, denoted by $D\omega$, is the  $(p+1,q)$ double form obtained by alternating the first $(p+1)$ arguments of the covariant differential $\nabla\omega$, precisely it is given by
\begin{equation*}
\begin{split}
D\omega(Y_0,Y_1,...,Y_p,Z_1,...,Z_q)=&{\rm alt}_{p+1}\nabla\omega(Y_0,Y_1,...,Y_p,Z_1,...,Z_q)\\
&=\sum_{i=0}^p(-1)^i  \nabla_{Y_i}\omega(Y_1,...,\hat{Y_i},...,Y_p,Z_1,...,Z_q).
\end{split}
\end{equation*}
Note that $D\omega$ coincides with the exterior derivative of  $\omega$ once seen as a vector valued $p$-form. Note that we are using here the opposite sign of the convention used in \cite{Kulkarni, Labbi-variational} for $D$.

\begin{proposition}
Let $x\in m$ and $(e_i)$ be a basis of the tangent space $T_xM$. Let $(e^i)$ denotes the dual basis, then for a double form $\omega$ the following holds at  $x$ 
\begin{equation}\label{nabla}
D\omega =\sum_{i=1}^n(e^i\otimes 1). \nabla_{e_i}\omega,
\end{equation}
where the above product is the exterior product of double forms.
\end{proposition}
\begin{proof}
Without loss of generality, we may assume that $\omega=\theta_1\otimes \theta_2$ is simple, then one has
\begin{equation*}
\begin{split}
\nabla\omega=&\sum_{i=1}^n e^i\otimes \nabla_{e_i}\omega =\sum_{i=1}^n e^i \otimes \bigl(\nabla_{e_i}\theta_1\otimes \theta_2+\theta_1\otimes  
\nabla_{e_i}\theta_2\bigr)\\
=&\sum_{i=1}^ne^i \otimes \nabla_{e_i}\theta_1\otimes \theta_2+\sum_{i=1}^ne^i \otimes \theta_1\otimes  \nabla_{e_i}\theta_2.
\end{split}
\end{equation*}
Now we apply the alternating operator to the first $(p+1)$ factors to get
\begin{equation*}
\begin{split}
D\omega=&{\rm alt}_{p+1}\nabla\omega=\sum_{i=1}^n(e^i \wedge \nabla_{e_i}\theta_1)\otimes \theta_2+\sum_{i=1}^n(e^i \wedge \theta_1)\otimes  \nabla_{e_i}\theta_2\\
=&\sum_{i=1}^n(e^i \otimes 1). \bigl(\nabla_{e_i}\theta_1\otimes \theta_2+\theta_1\otimes  
\nabla_{e_i}\theta_2\bigr)=\sum_{i=1}^n(e^i\otimes 1). \nabla_{e_i}\omega.
\end{split}
\end{equation*}

\end{proof}

Let now $(e_i)$ denotes a locally defined orthonormal frame. The divergence $\delta\omega$ of a $(p,q)$ double form $\omega$ is the $(p-1,q)$ double form given by
\begin{equation*}
\begin{split}
\delta\omega(Y_1,...,Y_{p-1},Z_1,...,Z_q)&:=-\sum_{i=1}^n(\nabla_{e_i}\omega)(e_i,Y_1,...,Y_{p-1},Z_1,...,Z_q)\\
&=-\sum_{i=1}^n\bigl(i_{e_i\otimes 1}\nabla_{e_i}\omega\bigr)(Y_1,...,Y_{p-1},Z_1,...,Z_q).
\end{split}
\end{equation*}
In particular, one has
\begin{equation}\label{delta}
\delta\omega=-\sum_{i=1}^ni_{e_i\otimes 1}\nabla_{e_i}\omega.
\end{equation}
We remark that $\delta\omega$ coincides with the divergence of  $\omega$ once seen as a vector valued $p$-form. Now, if we look at the same double form  $\omega$ as a vector valued $q$-form, we then define the maps
\begin{equation}\label{tildeDdelta}
\tD\omega =\bigl(D \omega^t\bigr)^t\,\,\, {\rm and}\,\,\,
\td\omega=\bigl(\delta \omega^t\bigr)^t.
\end{equation}

We list below for future reference some known properties of the operators $D,\tD, \td$ and $\delta$, see \cite{Kulkarni, Nasu, Labbi-variational}. Let $\omega,\omega_1$ be arbitrary $(p,q)$ double forms and $\omega_2$ be an $(r,s)$ double form, then one has
\begin{align}
D(&\omega_1.\omega_2)=D(\omega_1).\omega_2+(-1)^p\omega_1.D(\omega_2),  &    
\tD (\omega_1.\omega_2)&=\tD (\omega_1).\omega_2+(-1)^q\omega_1.\tD(\omega_2),&\label{eq2}\\
\delta\omega &=-\cc \tD-\tD \cc,&
\td\omega &=-\cc D-D\cc,&\label{eq4}\\
\delta(g^r&\omega)=(-1)^r\bigl(g^r\delta\omega+rg^{r-1}\tD \omega\bigr),&
\td (g^r\omega)&=(-1)^r\bigl(g^r\td\omega+rg^{r-1}D \omega\bigr), \,\,\,  r\geq 1,& \label{eq6}\\
D(&\str\omega)=(-1)^p\str\delta\omega, &            \tD(\str\omega)&=(-1)^q\str\td\omega &\label{eq7}\\
\str\, D&\,\str\,\omega=(-1)^{n(p+q+1)+q+1}\delta\omega,  & \str\, \tD\,\str\,\omega&=(-1)^{n(p+q+1)+p+1}\td\omega,  &\label{eq8}\\
\str\,\delta&\,\str\, \omega=(-1)^{n(p+q+1)+q}D\omega,&  \str\,\td\,\str\, \omega&=(-1)^{n(p+q+1)+p}\tD\omega. \label{eq9}
\end{align}
\begin{proof}
The identities in (\ref{eq2}) can be obtained directly from formulas (\ref{nabla}) and (\ref{tildeDdelta}). Next, we prove the first identity in (\ref{eq4}) using the last identity in formulas (\ref{2identities}) and the same notations as above 
$$\cc D\omega=\sum_{i=1}^n\cc\Bigl((e^i\otimes 1).\nabla_{e_i}\omega\Bigr)=\sum_{i=1}^n\Bigl[\iota_{1\otimes e_i}\nabla_{e_i}\omega-
(e^i\otimes 1).\cc\nabla_{e_i}\omega\Bigr]=-\td\omega-D\cc \omega.$$
The second identity follows directly from (\ref{eq4}) after using the identity $\delta \omega=(\td \omega^t)^t.$ To prove (\ref{eq6}) we use the following formula from \cite[formula (6)]{Labbi-double-forms}
$$\cc(g^r\omega)=g^r\cc\omega+r(n-p-q-r+1)g^{r-1}\omega,$$
and the identity (\ref{eq2}) to show that
\begin{align*}
\cc\tD(g^r\omega)&=(-1)^r\cc(g^r\tD\omega)=(-1)^r\Bigl[g^r\cc\tD\omega+r(n-p-q-r)g^{r-1}\tD\omega\Bigr],\\
\tD\cc(g^r\omega)&=\tD\Bigl[g^r\cc\omega+r(n-p-q-r+1)g^{r-1}\omega\Bigr]\\
&=(-1)^rg^r\tD\cc\omega+r(n-p-q-r+1)\bigl((-1)^{r-1}g^{r-1}\tD\omega\bigr).
\end{align*}
We then conclude using (\ref{eq4}). The second identity  can be proved in a similar way. We now prove the first identity (\ref{eq7}) using the same notations as in Proposition \ref{nabla} and the identity (\ref{firsteq}) as follows
$$D\str\, \omega=\sum_{i=1}^n(e^i\otimes 1).(\nabla_{e_i}\str\,\omega)=\sum_{i=1}^n(e^i\otimes 1).(\str\, \nabla_{e_i}\omega)=
\str\Bigl(\sum_{i=1}^n\iota_{e_i\otimes 1}\nabla_{e_i}\omega\Bigr)(-1)^{p+1}=(-1)^p\str\delta\omega.$$
The second identity in (\ref{eq7}) follows in a similar way. The identities in equations in  (\ref{eq8}) and (\ref{eq9}) are direct consequences of those in equation (\ref{eq7}). 
\end{proof}
As a consequence of the above identities (\ref{eq8}), (\ref{eq9}) and the identity (\ref{firsteq}) we prove the following 
\begin{proposition}
For a $(p,q)$ double form $\omega_1$ and a $(p+1,q)$ double form $\omega_2$, the divergence of the $1$-form $\str\, \bigl(\omega_1.\str\, \omega_2\bigr)$ satisfies
\begin{equation}
\delta\Bigl(\str\, \bigl(\omega_1.\str\, \omega_2\bigr)\Bigr)=-\langle D\omega_1,\omega_2\rangle+\langle\omega_1,\delta\omega_2\rangle.
\end{equation}
In particular, on a compact oriented Riemannian manifold endowed with the integral scalar product one has
\begin{equation}\label{adjointD-delta}
\langle D\omega_1,\omega_2\rangle=\langle\omega_1,\delta\omega_2\rangle.
\end{equation}
\end{proposition}
\begin{proof}
We use the identities (\ref{eq8}), (\ref{eq9}), (\ref{eq2}) and  (\ref{firsteq}) as follows
\begin{equation*}
\begin{split}
\delta\Bigl(\str\, \bigl(\omega_1.\str\, \omega_2&\bigr)\Bigr)=(-1)^{n-1} \str\str\, \delta\Bigl(\str\, \bigl(\omega_1.\str\, \omega_2\bigr)\Bigr)
=(-1)^{n-1}(-1)^n\str\, D\bigl(\omega_1.\str\, \omega_2\bigr)\\
=&-\str\, \bigl(D\omega_1.(\str\, \omega_2)+(-1)^p\omega_1.D\str\,\omega_2\bigr)=
-\langle D\omega_1, \omega_2\rangle -(-1)^p\Bigl(\omega_1.D\str\,\omega_2\bigr)\\
=&-\langle D\omega_1, \omega_2\rangle -(-1)^p(-1)^{p+1}\str\, \bigl(\omega_1.\delta\str\,\omega_2\bigr)=
-\langle D\omega_1,\omega_2\rangle+\langle\omega_1,\delta\omega_2\rangle.
\end{split}
\end{equation*}
The last claim follows from Stoke's theorem.
\end{proof}
\begin{remark}
We remark that formulas (\ref{eq8}) and (\ref{adjointD-delta}) correct a sign error in formulas (15), (16), (17) and (18) in \cite{Labbi-variational}.  These sign errors have no effect on the main results of \cite{Labbi-variational}.
\end{remark}
\begin{proposition}[Weitzenb\"{o}k formula for double forms]
For any double form $\omega$, one has
\begin{equation}
(D\delta+\delta D)(\omega)=\nabla^*\nabla \omega+\sum_{i,j=1}^n i_{e_i\otimes 1}\bigl((e^j\otimes 1).R_{e_ie_j}\omega\bigr),
\end{equation}
where $(e_i)$ is any orthonormal frame at the point in consideration,  $R_{e_ie_j}=\nabla^2_{e_je_i}\omega-\nabla^2_{e_ie_j}\omega$ and where the dot " ." denotes the exterior product of double forms.
\end{proposition}
\begin{proof}
Fix, $x\in M$ and let $(e_i)$ be a local orthonormal tangent frame field with the property that for all $i$  one has 
$\nabla_{e_i}=0$ at $x$. Let $(e^i)$ denotes the dual frame. Then at $x$, one can easily check that for a double form $\omega$ the following two identities hold
\begin{eqnarray}
\nabla_{e_j}(i_{e_i\otimes 1}\omega)&=&i_{e_i\otimes 1}\bigl(\nabla_{e_j}\omega\bigr),\\
\nabla_{e_j}\bigl((e^i\otimes 1).\omega\bigr)&=&(e^i\otimes 1).\nabla_{e_j}\omega.
\end{eqnarray}
Using the above identities and  equations \ref{nabla}, \ref{delta},  one has at the point $x$ the following
\begin{equation*}
\begin{split}
D\delta \omega &=-D\bigl[\sum_{i=1}^ni_{e_i\otimes 1}\nabla_{e_i}\omega\bigr]
=-\sum_{i=1}^n\bigl[ D(i_{e_i\otimes 1}\nabla_{e_i}\omega)\bigr]\\
&=-\sum_{i,j=1}^n(e^j\otimes 1).\nabla_{e_j} \bigl[ i_{e_i\otimes 1}\nabla_{e_i}\omega\bigr]
=-\sum_{i,j=1}^n(e^j\otimes 1). i_{e_i\otimes 1}\nabla^2_{e_j,e_i}\omega\\
&=-\sum_{i=1}^n(e^i\otimes 1). i_{e_i\otimes 1}\nabla^2_{e_i,e_i}\omega-
\sum_{i\not=j}(e^j\otimes 1). i_{e_i\otimes 1}\nabla^2_{e_j,e_i}\omega\\
=& -\sum_{i=1}^n\nabla^2_{e_i,e_i}\omega+\sum_{i=1}^n i_{e_i\otimes 1}\bigl( (e^i\otimes 1).\nabla^2_{e_i,e_i}\omega\bigr)+\sum_{i\not=j} i_{e_i\otimes 1}\bigl( (e^j\otimes 1).\nabla^2_{e_j,e_i}\omega\bigr).
\end{split}
\end{equation*}
Where we used the fact that the left interior product $i_{e_i\otimes 1}$  operates on the exterior algebra of double forms as an anti-derivation. 
In a similar way one has as well,
\begin{equation*}
\begin{split}
\delta D \omega &=\delta\bigl[\sum_{i=1}^n(e^i\otimes 1).\nabla_{e_i}\omega\bigr]
=-\sum_{i,j=1}^ni_{e_j\otimes 1} \nabla_{e_j} \bigl[ (e^i\otimes 1).\nabla_{e_i}\omega\bigr]\\
& =-\sum_{i,j=1}^ni_{e_j\otimes 1} \bigl[ (e^i\otimes 1).\nabla^2_{e_j,e_i}\omega\bigr]
=-\sum_{i=1}^ni_{e_i\otimes 1} \bigl[ (e^i\otimes 1).\nabla^2_{e_i,e_i}\omega\bigr]
-\sum_{i\not= j}^ni_{e_j\otimes 1} \bigl[ (e^i\otimes 1).\nabla^2_{e_j,e_i}\omega\bigr].
\end{split}
\end{equation*}
Consequently, we get 
\begin{equation*}
\begin{split}
D\delta\omega+\delta D\omega=&-\sum_{i=1}^n\nabla^2_{e_i,e_i}\omega+
\sum_{i\not= j}^ni_{e_i\otimes 1} \bigl[ (e^j\otimes 1).\bigl(\nabla^2_{e_j,e_i}\omega-\nabla^2_{e_i,e_j}\omega\bigr) \bigr]\\
=& \nabla^*\nabla \omega+\sum_{i,j=1}^n i_{e_i\otimes 1}\bigl[(e^j\otimes 1).R_{e_ie_j}\omega\bigr].
\end{split}
\end{equation*}
\end{proof}
Let's denote by $\Weit(w)$ the curvature term of order $0$ in the above Weitzenb\"{o}k formula. We first remark that
\begin{equation}\label{Weit-left}
\begin{split}
\Weit(\omega)=&\sum_{i,j=1}^n i_{e_i\otimes 1}\mu_{e^j\otimes 1}R_{e_ie_j}(\omega)\\
=& \sum_{i<j}\Bigl[ i_{e_i\otimes 1}\mu_{e^j\otimes 1}-i_{e_j\otimes 1}\mu_{e^i\otimes 1}\Bigr]R_{e_ie_j}(\omega)\\
=&- \sum_{i<j} (\rho_{ij})_\ell R_{e_ie_j}(\omega),
\end{split}
\end{equation}
where, for a double form $\theta$, the map $\mu_{\theta}$ denotes the left multiplication by $\theta$ in the exterior algebra of double forms, the operator $(\rho_{ij})_\ell$ is as above. Next, a direct application of part (2) of Proposition \ref{4parts-proposition} and of the formula in  Proposition \ref{rij-derivations} yields the following result
\begin{proposition}\label{Weit-sum}
Let $(M,g)$ be a Riemannian $n$-manifold and $\omega$ be a $(p,q)$ double form. The Weitzenb\"ock operator of $(M,g)$ equals
\begin{equation}
\Weit(\omega)=\sum_{i<j}\Bigl[-\omega \circ \frac{g^{p-1}R_{ij}}{(p-1)!} \circ \frac{g^{p-1}\rho_{ij}}{(p-1)!}+\frac{g^{q-1}R_{ij}}{(q-1)!} \circ \omega \circ \frac{g^{p-1}\rho_{ij}}{(p-1)!}\Bigr].
\end{equation}
In the above formula, $\circ$ denotes the composition of double forms and $g^{-1}$ denotes the null double form.
\end{proposition}
\begin{corollary}\label{Weitz-gp}
Let $(M,g)$ be a Riemannian $n$-manifold and $\omega$ be a $(p,p)$ double form. Then for each $1\leq p\leq n$,  the Weitzenb\"ock operator of $(M,g)$ satisfies
\begin{flalign}
\Weit(\frac{g^p}{p!})=&0.\\
\langle \Weit(\omega), \frac{g^p}{p!}\rangle =&0.
\end{flalign}
\end{corollary}
\begin{proof}
The first assertion results immediately from the above proposition and the fact that $\frac{g^p}{p!}$ is the identity for the composition product of $(p,p)$ double forms. The second one follows in a similar way if one uses a basis of $2$ forms that diagonalizes the curvature operator as below.
\end{proof}

Recall that $\rho_{ij}=\rho(e^i\wedge e^j)$ and $R_{ij}=\rho(R^\flat(e_i\wedge e_j))$, where $(e_i)$ is an orthonormal basis of the tangent space and $(e^i)$ is the dual basis. The above sum in formula (\ref{Weit-sum}) is independent of the choice of the basis on two forms, so let ($E_\alpha$) be an orthonormal basis diagonalizing the Riemann curvature symmetric double form $R$ and denote by $\lambda_\alpha$ the corresponding eigenvalues. Then $R=\sum_{\alpha}\lambda_\alpha E_{\alpha}\otimes E_{\alpha}$ and one immediately gets
\begin{equation}\label{Weit-lambda}
\Weit(\omega)=\sum_{\alpha}\lambda_{\alpha} \Bigl[-\omega \circ \frac{g^{p-1}\rho(E_\alpha)}{(p-1)!} \circ \frac{g^{p-1}\rho(E_\alpha)}{(p-1)!}+\frac{g^{q-1}\rho(E_\alpha)}{(q-1)!} \circ \omega \circ \frac{g^{p-1}\rho(E_\alpha)}{(p-1)!}\Bigr].
\end{equation}
Using the identities \ref{Greub-Vanstone} and \ref{2-identiries} one has
\[\sum_{i<j}\Bigl[\frac{g^{p-1}R_{ij}}{(p-1)!} \circ \frac{g^{p-1}\rho_{ij}}{(p-1)!}\Bigr]=\sum_{i<j}
\Bigl[\frac{g^{p-1}(R_{ij}\circ \rho_{ij})}{(p-1)!} +\frac{g^{p-2}R_{ij}\rho_{ij}}{(p-2)!}\Bigr]=-\frac{g^{p-1}\Ric}{(p-1)!}+2 \frac{g^{p-2}R}{(p-2!} \]
The following relation follows from the identity \ref{sharp-composition}
\[ \sum_{\alpha}\lambda_{\alpha} \Bigl[\frac{g^{q-1}\rho(E_\alpha)}{(q-1)!} \circ \omega \circ \frac{g^{p-1}\rho(E_\alpha)}{(p-1)!}\Bigr]
=\frac{-1}{4}\sum_{\alpha}\lambda_{\alpha} (E_{\alpha}\otimes E_{\alpha})\# \omega=\frac{-1}{4} R\# \omega.\]
We have therefore proved the following index free formula for the Weitzenb\"ock curvature operator on double forms
\begin{proposition}\label{Weit-sharp-form}
For a $(p,q)$ double form $\omega$ with $p,q\geq 0$ one has
\begin{equation}
\Weit(\omega)=\omega \circ \Bigl( \frac{g^{p-1}\Ric}{(p-1)!}-2 \frac{g^{p-2}R}{(p-2!} \Bigr)-\frac{1}{4} R\# \omega,
\end{equation}
where we use the convention $g^{-1}=0$.
\end{proposition}
Recall that in case $p=0$ or $q=0$, one has $R\# \omega=0$. Below, we will explain how we recover several known Weitzenb\"ock formulas in the literature as a special case of the above general result.
\begin{enumerate}
\item The case $q=0$ and $p=1$. Let $\theta$ be a 1-form and identify it to  $\omega=\theta\otimes 1$ a $(1,0)$ double form. Then we recover the famous Weitzenb\"ock formula for 1-forms as follows
\begin{equation}
\begin{split}
\Weit(\omega)=&\omega\circ \Ric=\sum_{i,j}(\theta\otimes 1)\circ \Ric(e_i,e_j)e^i\otimes e^j\\
=& \sum_{i,j}\Ric(e_i,e_j)\langle e^j,\theta\rangle e^i\otimes 1= \sum_{i,j}\Ric^{\flat}(e^i,e^j)\langle e^j,\theta\rangle e^i\otimes 1\\
=&  \sum_{i}\Ric^{\flat}(e^i,\theta)    e^i\otimes 1=\sum_{i}\langle \Ric^{\flat}(\theta)  , e^i\rangle e^i\otimes 1=\Ric^{\flat}(\theta)\otimes 1.
\end{split}
\end{equation}
\item The case $q=0$ and $p\geq 2$. In this case we have
\begin{equation}\label{Weit-p-forms}
\Weit(\omega)=\omega \circ \Bigl( \frac{g^{p-1}\Ric}{(p-1)!}-2 \frac{g^{p-2}R}{(p-2!} \Bigr).
\end{equation}
As in the above case we identify a $p$-form $\theta$ with   $\omega=\theta\otimes 1$ a $(p,0)$ double form. Then one can recover the classical Weitzenb\"ock formula for $p$-forms in its compact format as in \cite{Bourguignon} and \cite{Labbi-nachr}.
\item The case $p=q=1$. Using Proposition \ref{sharp-forms} we recover formula (4.1) of Bourguignon \cite{Bourguignon}:
\begin{equation}
\Weit(h)=h \circ \Ric -i_{h^{\sharp}}R.
\end{equation}
In fact, for $h=\sum_{i,j}h(e_i,e_j)e^i\otimes e^j$ one has
\[ i_{h^{\sharp}}R(X,Y)=\sum_{i,j}h(e_i,e_j)i_{e_i\otimes e_j}R(X,Y)=\sum_{i,j}h(e_i,e_j)Re_i,X,e_j,Y)=\mathring{R}(h)(X,Y).\]
\item The case $p=q=2$. Using Proposition \ref{sharp-forms} we recover formula (4.2) of Bourguignon \cite{Bourguignon}:
\begin{equation}\label{Bourg}
\Weit(\omega)=\omega \circ (g\Ric -2R)-\frac{1}{4}R\#\omega.
\end{equation}
In fact, using Proposition \ref{sharp-forms} we get
\begin{equation}
\begin{split}
\frac{1}{4}(R \# &\omega )(x,y,z,t)=\sum_{j,k=1}^ni_{e_j\otimes e_k}\omega. i_{e_j\otimes e_k}R(x,y,z,t)\\
=& \sum_{j,k=1}^n \Bigl[ i_{e_j\otimes e_k}(x,z)\omega. i_{e_j\otimes e_k}R(y,t)+i_{e_j\otimes e_k}(y,t)\omega. i_{e_j\otimes e_k}R(x,z)\\
&\,\,\,\,\,\,\,\,\,\,\,\,\,\,\,\,\,\,\, -i_{e_j\otimes e_k}(x,t)\omega. i_{e_j\otimes e_k}R(y,z)-i_{e_j\otimes e_k}(y,z)\omega. i_{e_j\otimes e_k}R(x,t)\Bigr]\\
=&\sum_{j,k=1}^n \Bigl[ \omega(e_j,x,e_k,z)R(e_j,y,e_k,t)+\omega(e_j,y,e_k,t)R(e_j,x,e_k,z)\\
&\,\,\,\,\,\,\,\,\,\,\,\,\,\,\,\,\,\, -\omega(e_j,x,e_k,t)R(e_j,y,e_k,z)-\omega(e_j,y,e_k,z)R(e_j,x,e_k,t)\Bigr].
\end{split}
\end{equation}
\end{enumerate}
As an application of the above formula, we derive the following known formula for the evolution of the Riemann curvature tensor under the Ricci flow in terms of the connection Laplacian $\nabla^*\nabla$.
\begin{corollary}\label{sharp-use}
The Riemann curvature tensor $R$, seen as a $(2,2)$ double form, evolves under the Ricci flow $g'(t)=-2\Ric(g_t)$ according to the following formula
\[R'(t)=-\nabla^*\nabla-R\circ g\Ric-g\Ric\circ R+2R\circ R-\frac{1}{4}R\# R.\]
\end{corollary}
\begin{proof}
It results from the proof of Lemma 4.1 in \cite{Labbi-variational} with $h=-2\Ric$ that
\[R'(t)(x,y,z,u)=D\tD\Ric(x,y,z,u)-\Ric(R(x,y)z,u)+\Ric(R(x,y)u,z).\]
Then formula (\ref{hR}) together with the identity $\delta R=-\cc\tD R-\tD \cc R=-\tD \cc R$ show that
\[R'(t)=-D\delta R-(g\Ric \circ R).\]
Recall that $DR=\tD R=0$ by the second Bianchi identity. Therefore,  the above Weitzenb\"ok formula and formula (\ref{Bourg}) lead to the desired formula.
\end{proof}
\begin{corollary}
For the standard sphere $S^n$ of curvature $1$, the Weitzenb\"ock curvature term for $(p,q)$ double forms is given by
\begin{equation}\label{Weit-sphere}
\Weit(\omega)=\bigl[ p(n-p)+q\bigr] \omega-g\cc \omega -\widetilde{\mathfrak{S}}\mathfrak{S}\omega.
\end{equation}
In particular, for a $(p,p)$ double form satisfying the first Bianchi identity, one has
\begin{equation} \Weit(\omega)=\bigl[ p(n-p+1)\bigr] \omega-g\cc \omega.\end{equation}
\end{corollary}
\begin{proof}
For  the standard sphere $S^n$ of curvature $1$, one has $R=1/2g^2$ and $\Ric=(n-1)g$, the above proposition and Corollary \ref{g-square-sharp} show that
\begin{equation*}
\begin{split}
\Weit(\omega)=& \omega \circ \Bigl( \frac{g^{p-1}\Ric}{(p-1)!}-2 \frac{g^{p-2}R}{(p-2!} \Bigr)-\frac{1}{4} R\# \omega\\
=& \omega \circ \frac{g^p}{p!}p(n-p)-\frac{1}{8} g^2\# \omega=p(n-p)\omega-\frac{1}{4}\bigl( 4g\cc \omega + \widetilde{\mathfrak{S}}\mathfrak{S}\omega-4q\omega\bigr).
\end{split}
\end{equation*}
\end{proof}

\begin{lemma}\label{B-alpha-estimations}
\begin{enumerate}
\item Let $E\in \wedge^2V^*$ be a 2-form with rank $2r$ and let $\rho(E)$ be the corresponding bilinear form as above. 
 Then the double form $-\frac{g^{p-1}\rho(E)}{(p-1)!}\circ \frac{g^{p-1}\rho(E)}{(p-1)!}$ is symmetric and nonnegative. Furthermore, one has
\begin{equation}
-\frac{g^{p-1}\rho(E)}{(p-1)!}\circ \frac{g^{p-1}\rho(E)}{(p-1)!}\leq \Min\{p,r\}\frac{g^p}{p!}.
\end{equation}
\item Let $(E_i)$ be an orthonormal basis of $\wedge^2V^*$, then the following holds
\begin{equation}
-\sum_{i=1}^{n(n-1)/2}\frac{g^{p-1}\rho(E_i)}{(p-1)!}\circ \frac{g^{p-1}\rho(E_i)}{(p-1)!}=p(n-p)\frac{g^p}{p!}.
\end{equation}

\end{enumerate}
\end{lemma}
\begin{proof}
The double form under consideration is symmetric as it is trivially equal to its transpose.The rank of $E$ equals the rank of $\rho(E)$ as a bilinear form. The later being skew symmetric, then it has $2r$ pure imaginary non zero eigenvalues, say, $\pm i\lambda_{1},..., \pm i \lambda_{r}$, the remaining eigenvalues are zero. We suppose $\lambda_{1}\geq \lambda_{2}\geq ...\geq \lambda_{r}>0.$\\
Recall that the linear operator associated to the double form $\frac{g^{p-1}\rho(E)}{(p-1)!}$ is the extension to $\wedge^pV^*$ by derivations of the linear operator associated to $\rho(E)$. Therefore, the eigenvalues of $\frac{g^{p-1}\rho(E)}{(p-1)!}$ are sums of $p$ eigenvalues of $\wedge^pV^*$, that are $i\bigl( \pm \lambda_{i_1}\pm ... \pm \lambda_{i_k}\bigr)$, with $1\leq i_1<i_2<...<i_k\leq r$, $k\leq \min\{p,r\}$, the remaining eigenvalues are zero.\\
Consequently, the eigenvalues of $\frac{g^{p-1}\rho(E)}{(p-1)!}\circ \frac{g^{p-1}\rho(E)}{(p-1)!}$ are zeros or $-\bigl( \pm \lambda_{i_1}\pm ... \pm \lambda_{i_k}\bigr)^2$. We complete the proof as follows 
\begin{equation}
 \bigl( \pm \lambda_{i_1}\pm ... \pm \lambda_{i_k}\bigr)^2 \leq k\bigl(  \lambda_{i_1}^2+ ... + \lambda_{i_k}^2\bigr)\leq k\bigl(  \lambda_{1}^2+ ... + \lambda_{r}^2\bigr)\leq k  \leq \Min\{p,r\}.
\end{equation}
In the last steps we used Newton's inequality and the fact that the eigenvalues of $\rho(E)\circ \rho(E)$ come in pairs and hence
\[ \lambda_{1}^2+ ... + \lambda_{r}^2=\frac{1}{2}||\rho(E)||^2=||E||^2=1.\]
This completes the proof of part (1). Next, we prove the second part.  One quick way to prove it is to remark that the left hand side of the equation is nothing but the Weitzenb\"ock curvature of the unit standard sphere ($R=1/2g^2$). We provide below a proof that uses the identities \ref{Greub-Vanstone} and \ref{2-identiries} as follows
\begin{equation*}
\begin{split}
\sum_{i=1}^{n(n-1)/2}\frac{g^{p-1}\rho(E_i)}{(p-1)!}&\circ \frac{g^{p-1}\rho(E_i)}{(p-1)!}=
\sum_{i=1}^{n(n-1)/2}
\Bigl[\frac{g^{p-1}(\rho(E_i)\circ \rho(E_i))}{(p-1)!} +\frac{g^{p-2}\rho(E_i)\rho(E_i)}{(p-2)!}\Bigr]\\
&=-\frac{g^{p-1}(n-1)g}{(p-1)!}+2 \frac{g^{p-2}g^2}{2(p-2)!} =-p(n-p)\frac{g^p}{p!}.
\end{split}
\end{equation*}
\end{proof}
\noindent
The following key algebraic lemma is a re-formulation of an inequality due to Petersen-Wink \cite{Pet-Wink}  
\begin{lemma}\label{weighted-sum}
Let $\lambda_1\leq ...\leq \lambda_n$ be $n$real numbers and $w_1, ...,w_n$ be nonnegative real numbers such that $M=\max\{w_1,...,w_n\}>0$. Let 
$S=\sum_{i=1}^nw_i$ and  $k=\lfloor\frac{S}{M}\rfloor$. Then the following inequality holds
\[\sum_{i=1}^n\lambda_iw_i\geq \frac{S}{k}\sum_{i=1}^k\lambda_i.\]
In particular, one has
\begin{enumerate}
\item If $\sum_{i=1}^k\lambda_i\geq 0$ then $\sum_{i=1}^n\lambda_iw_i\geq 0.$
\item If $\sum_{i=1}^k\lambda_i>0$ then $\sum_{i=1}^n\lambda_iw_i> 0.$
\end{enumerate}
\end{lemma}
\begin{proof}
\begin{equation*}
\begin{split}
\sum_{i=1}^n\lambda_iw_i=&\sum_{i=1}^k\lambda_iw_i+\sum_{i=k+1}^n\lambda_iw_i\geq \sum_{i=1}^k\lambda_iw_i+\lambda_{k+1}\sum_{i=k+1}^n w_i\\
=& \sum_{i=1}^k\lambda_iw_i+\lambda_{k+1}\bigl( \sum_{i=1}^n w_i-\sum_{i=1}^k w_i\bigr) = \sum_{i=1}^k(\lambda_i-\lambda_{k+1}) w_i+\lambda_{k+1}S\\
\geq & M \sum_{i=1}^k(\lambda_i-\lambda_{k+1}) +\lambda_{k+1}S\geq  \frac{S}{k}  \sum_{i=1}^k(\lambda_i-\lambda_{k+1})+\lambda_{k+1}S= \frac{S}{k}\sum_{i=1}^k\lambda_i. 
\end{split}
\end{equation*}
\end{proof}
Recall that the rank of a 2-form $\alpha\in \wedge^2V^*$ is the minimum number of co-vectors in $V^*$ in terms of which $\alpha$ can be expressed.
\begin{definition}
\begin{itemize}
\item[a)] We say that the Riemann curvature tensor $R$  of a Riemannian $n$-manifold has {\emph purity rank} $\leq 2r$, for some $r$ with $2\leq 2r\leq n$,  if at each $m\in M$,  there exists  an orthonormal basis of eigenvectors of the curvature operator $R$ each of which has rank $\leq 2r$.
\item[b)] We say that the Riemann curvature tensor $R$ is  $k$-positive (resp.  $k$-non-negative) if the sum of the lowest $k$ eigenvalues of the associated curvature operator (counted with multiplicity) is positive (resp. non-negative).
\end{itemize}
\end{definition}
The Riemann tensor has purity rank $\leq 2$ if and only if $R$ is pure in the sense of Maillot \cite{Maillot-a, Maillot-t}. The standard round sphere $S^n$, $n\geq 2$, has purity rank $\leq 2$ and the complex projective space $\mathbb{C}P^n$ with the Fubiny-Study metric has purity rank $\leq 4$ for $n\geq 2$. We remark that the property of having purity rank $\leq 2k$, for  fixed $k$, is stable under Riemannian products.

We are now ready to prove  Theorem A stated in the introduction
\begin{thm-a}
Let $(M,g)$ be a closed and connected Riemannian manifold of dimension $n\geq 3$ and $p$ be a positive integer such that $2p\leq n$. Suppose the Riemann curvature tensor $R$ has purity rank $\leq 2r$ and $R$ is $k$-positive (resp. $k$-nonnegative) with
 $1\leq k\leq   \frac{p(n-p)}{\Min\{p,r\}}$,
then the Betti numbers $b_p$ and $b_{n-p}$ of $M$ vanish (resp. every harmonic $p$-form on $M$ is parallel).
\end{thm-a}
Note that in the case of minimal purity where $r=1$, the minimum of $p$ and $r$ is $1$ and therefore we recover a result of Maillot \cite{Maillot-a, Maillot-t}, namely, the vanishing of $b_p$ under the condition of $R$ being $p(n-p)$-positive. The condition of maximal purety, that is $2r\geq n-1$, is a free condition satisfied for any Riemannian manifold. In such a case the minimum of $p$ and $r$ is $p$ and therefore we recover a recent result of Petersen and Wink \cite{Pet-Wink, Bet-Good}, namely, the vanishing of $b_p$ under the condition of $R$ being $(n-p)$-positive.
\begin{proof}
Let $m\in M$,  $\omega$ be a $p$-form on $M$ and $(E_{\alpha})$ be an orthonormal basis of eigenvectors for $R$ at $m$ each of which has rank $\leq 2r$. Let $B_{\alpha}=-\frac{g^{p-1}\rho(E_\alpha)}{(p-1)!} \circ \frac{g^{p-1}\rho(E_\alpha)}{(p-1)!}$.\\
Using formula \ref{Weit-lambda}, the curvature term in the Weitzenb\"ock formula in the case $q=0$ is equal to
\begin{equation}
\begin{split}
\langle \Weit(\omega\otimes 1), \omega\otimes 1\rangle&=\langle (\omega\otimes 1) \circ \sum_{\alpha}\lambda_{\alpha} B_{\alpha}, \omega\otimes 1\rangle\\
&= \langle \sum_{\alpha}\lambda_{\alpha} B_{\alpha},(1\otimes \omega)\circ (\omega\otimes 1)\rangle=\langle \sum_{\alpha}\lambda_{\alpha} B_{\alpha},\omega\otimes \omega\rangle\\
&=\sum_{\alpha}\lambda_{\alpha} B_{\alpha}(\omega^{\sharp}, \omega^{\sharp}).
\end{split}
\end{equation}
Lemma \ref{B-alpha-estimations} shows that
\[ 0\leq B_{\alpha}(\omega^{\sharp}, \omega^{\sharp})\leq \min\{p,r\}||\omega^{\sharp}||^2 \,\, {\rm and}\, \,\,  \sum_{\alpha}B_{\alpha}(\omega^{\sharp}, \omega^{\sharp})=p(n-p)||\omega^{\sharp}||^2.\]
At a point where $\omega=0$ one trivially has $\sum_{\alpha}\lambda_{\alpha} B_{\alpha}(\omega^{\sharp}, \omega^{\sharp})=0$. Now, if at some point $\omega\not=0$, the condition $\sum_{i=1}^k\lambda_i\geq 0$ (resp. $>0$) implies 
$\sum_{\alpha}\lambda_{\alpha} B_{\alpha}(\omega^{\sharp}, \omega^{\sharp})\geq 0$ (resp. $>0$). Consequently, in the case if $\sum_{i=1}^k\lambda_i\geq 0$ at every point of $M$, a harmonic $p$-form $\omega$ must be parallel and in the case $\sum_{i=1}^k\lambda_i> 0$ at every point of $M$, a harmonic $p$-form $\omega$ must be everywhere zero.
\end{proof}

\subsection{Hodge-de Rham Laplacian on double forms}
We define the Hodge-de Rham Laplacian of a double form $\omega$ to be 
\[ \Delta (\omega):= (D\delta +\delta D)(\omega).\]
The Weitzenb\"ock formula shows that the principal part of the operator $\Delta$ is the connection Laplacian $\nabla^*\nabla$, consequently it is an elliptic operator and therefore its kernel is a finite dimensional vector space in case $M$ is a compact manifold. We shall call the kernel of $\Delta$ the space of Harmonic double forms.
We note that a double form $\omega$ on a compact manifold is harmonic if and only if it is closed and co-closed, that is $D\omega=0$ and $\delta\omega=0$. In fact, using the integral scalar product one has see formula (\ref{adjointD-delta})
$$\langle\Delta\omega,\omega\rangle=\langle \delta\omega,\delta\omega\rangle+\langle D\omega,D\omega\rangle.$$
Recall that \cite{Labbi-double-forms}, the double Hodge star operator  $*$ operates on double forms in the following way
$$*(\theta_1\otimes \theta_2)=*\theta_1\otimes *\theta_2.$$
\begin{proposition}\label{Hodge-properties}
\begin{enumerate}
\item The Hodge-de Rham Laplacian on double forms commutes with the double Hodge star operator. Precisely, for any double form  $\omega$ one has
\[  \Delta(*\omega)=*\Delta(\omega).\]
In particular, $\omega$ is harmonic if and only if $*\omega$ is harmonic.
\item The Weitzenb\"ock curvature operator on double forms satisfies
\[ \Weit(*\omega)=*\Weit(\omega).\]

\end{enumerate}
\end{proposition}
\begin{proof}
Formula (\ref{eq7})  shows that $D*=(-1)^{p}*\delta$ and $\delta *=(-1)^{p+1}*D$ on $(p,q)$ double forms. Composing, we easily  see that $D\delta*=*\delta D$ and $\delta D*=*D\delta$. 
This completes the proof of the first part. 

For the second part, we first remark that for a skew-symmetric $(1,1)$ double form, formula (25) in \cite{Labbi-Bel} shows that for $1\leq k\leq n-1$
$$*\frac{g^{k-1}h}{(k-1)!}=\frac{g^{n-k-1}h}{(n-k-1)!}.$$
Formula (25) of this paper together with Proposition 4.4 in \cite{Labbi-Bel} show without difficulties that $ \Weit(*\omega)=*\Weit(\omega)$.
\end{proof}
The Laplacian $\Delta$ do not in general send a symmetric double form onto a symmetric double form. The following proposition explains when this is possible
\begin{proposition}
For a Riemannian $n$-manifold $(M,g)$, the Laplacian $\Delta$ sends  symmetric $(p,p)$ double forms onto  symmetric $(p,p)$ double forms if and only if one of the following conditions holds
\begin{enumerate}
\item $p=1$ and $(M,g)$ is Einstein.
\item $n=2p$ and $(M,g)$ is conformally flat.
\item $p>1$, $n\not=2p$ and $(M,g)$ has constant sectional curvature.
\end{enumerate}

\end{proposition}

\begin{proof}
Let $\omega$ be a $(p,p)$ symmetric double form. The Weitzenb\"ock formula shows that
\[ (\Delta \omega)^t=\nabla^*\nabla\omega +(\Weit(\omega))^t.\]
Let $\mathcal{N}_p=\frac{g^{p-1}\Ric}{(p-1)!}-2 \frac{g^{p-2}R}{(p-2)!}$ for $p\geq 2$ and $\mathcal{N}_1=\Ric$. Proposition \ref{Weit-sharp-form} shows that the condition $(\Weit(\omega))^t=\Weit(\omega)$ is equivalent to the commutation property $\mathcal{N}_p\circ \omega =\omega\circ \mathcal{N}_p$.\\
Consider an orthonormal basis $(\theta_i)$ of $p$-forms  that diagonalizes $\mathcal{N}_p$ so that $\mathcal{N}_p=\sum_i\mu_i\theta_i\otimes \theta_i$, $\mathcal{N}_p$ commutes with all symmetric  $(p,p)$ double form if and only if for all $i,j$ one has
\[\bigl(\sum_k\mu_k\theta_k\otimes \theta_k\bigr)\circ (\theta_i\otimes \theta_j+\theta_j\otimes \theta_i)=(\theta_i\otimes \theta_j+\theta_j\otimes \theta_i)\circ \bigl(\sum_k\mu_k\theta_k\otimes \theta_k\bigr).\]
This can be seen to be equivalent to the fact that all eigenvalues of $\mathcal{N}_p$ are equal, say to a constant $\mu\in \mathbb{R}$. If $p=1$, this is equivalent to the fact that $\Ric=\mu g$, that is $(M,g)$ is Einstein. If $p\geq 2$, the condition 
\[ \frac{g^{p-1}\Ric}{(p-1)!}-2 \frac{g^{p-2}R}{(p-2)!}=\mu \frac{g^p}{p!},\]
is equivalent to $(M,g)$ conformally flat if $n=2p$ and to constant sectional curvature otherwise, see \cite[Corollary 4.3]{Labbi-nachr}.
\end{proof}
\begin{proposition}\label{proposition5.3} Let $\omega$ be a  $(p,p)$ double form and let $\omega=\sum_{i=0}^ng^{p-i}\omega_i$ be the orthogonal decomposition of $\omega$, see section \ref{ortho-decomp}.
\begin{enumerate}
\item $\omega$ is harmonic if and only if $\omega$ and $\cc \omega$ are both closed double forms.
\item $\omega$ is harmonic if and only if all the components $\omega_i$ for $i=0,1,...,p$ are harmonic. 
\item If $D\omega=0$ and $\delta \omega_p=0$, then the top trace free component $\omega_p$ is harmonic.
\end{enumerate}
\end{proposition}
\begin{proof}
Let $\omega$ be harmonic, then $D\cc \omega=-\tilde{\delta}\omega-\cc D\omega=0$. Conversely, let $\omega$ and $\cc \omega$ be both closed. Then $\delta \omega=-\cc \tilde{D}\omega-\tilde{D}\cc \omega=0$ and therefore $\omega$ is harmonic. This proves the first part.\\
Next, if all the components $\omega_i$ are harmonic then trivially $\omega$ it self is harmonic. Conversely,  if $\omega$ is harmonic then all the contractions $\cc \omega, \cc^2 \omega, \cc^3 \omega,...$ are harmonic by the first part. Theorem 3.7 in \cite{Labbi-double-forms} shows that each component $\omega_i$ is a linear combination with constant coefficients of the terms of the form $g^r\cc^k \omega$ which are all harmonic. This completes the proof of the second part.\\
According to formula (25) in \cite{Labbi-double-forms}, one has 
$$*\omega_p=(-1)^p\frac{1}{(n-2p)!}g^{n-2p}\omega_p.$$
 Hence, $D(*\omega_p)=(-1)^{n-p}\frac{1}{(n-2p)!}g^{n-2p}D\omega_p.$ Formula ? asserts that $D(*\omega_p)=(-1)^p*\delta\omega_p=0$, we conclude immediately that $g^{n-2p}D\omega_p=0$. \\
From another side we have $\omega =\omega_p+gA$, then $0=D\omega_p-gDA$ and hence $g^{n-2p+1}DA=0$. As $p+p-1+n-2p+1=n<n+1$, 
one has $DA=0$ by \cite[Proposition 2.3]{Labbi-double-forms} and then $D\omega_p=0$. This completes the proof of the proposition.
\end{proof}

\begin{proposition}[Schur's theorem]\label{Schur}
Let $2\leq p\leq n-1$ and let $\omega$ be a $(p,p)$ double form such that $\tilde{D}\omega=0$. If at each point of $x\in M$, one has $\cc\,  \omega =\lambda g^{p-1}$, where $\lambda=\lambda(x)\in \mathbb{R}$, then $\lambda$ is constant on $M$.
\end{proposition}
\begin{proof}
Using formula (\ref{eq6}), we have
$$\delta\cc \omega =\delta (\lambda g^{p-1})=(-1)^{p-1}\{g^{p-1}\delta\lambda-(p-1)g^{p-2}\tilde{D}\lambda\}=(p-1)\tilde{D}(\lambda g^{p-2}).$$
On the other hand we have $\delta\cc \omega=-1/2\tilde{D}\cc^2\omega$, therefore,
$$\tilde{D}\bigl((p-1)\lambda g^{p-2}-1/2\cc^2\omega\bigr)=0.$$
A direct computation shows that $\cc^2\omega=\cc(\lambda g^{p-1})=(n-p+2)(p-1)\lambda g^{p-2}$, hence the above equation yields
$$\tilde{D}\bigl((p-1)\lambda g^{p-2}-\frac{(p-1)(n-p+2)}{2}\lambda g^{p-2}\bigr)= (p-1) (\tilde{D}\lambda). \bigl(1-\frac{n-p+2}{2}\bigr)g^{p-2}=0.$$
Since $n\not=p$ one then has $\tilde{D}\lambda =0$ by \cite[Proposition 2.3]{Labbi-double-forms} and hence $\lambda$ is constant.
\end{proof}
\section{Lichn\'erowicz Laplacian and proofs of Theorems  B and C}
What we did so far, is to look at a $(p,q)$ double form as a vector valued $p$-form and extended the definition of exterior derivative to the map $D$ etc... Alternatively, we can look at a $(p,q)$ double form as a vector valued $q$-form and reformulate all our above analysis. For instance the extension of the exterior derivative will be denoted by $\tilde{D}$ and refereed to  as the adjoint Bianchi sum \cite{Kulkarni}. We remark that
$$\tilde{D}(\omega)=\bigl(D(\omega^t)\bigr)^t,$$
similarly, $\tilde{\delta}(\omega)=\bigl(\delta(\omega^t)\bigr)^t$. We then define the adjoint Laplacian to be
\[ \tilde{\Delta}=\tilde{D}\tilde{\delta}+\tilde{\delta}\tilde{D}.\]
It is not difficult to see that $\tilde{\Delta}(\omega)=\bigl(\Delta(\omega^t)\bigr)^t$. Hence using formula (\ref{Weit-left}) we deduce easily a Weitzenb\"ock formula for the adjoint Laplacian as follows
\[ \tilde{\Delta}(\omega)=\nabla^*\nabla(\omega)-\sum_{i<j}(\rho_{ij})_rR_{e_ie_j}(\omega),\]
where we used the notations and results of Proposition \ref{ext-formulas} and the fact that $\bigl( R_{e_ie_j}(\omega^t)\bigr)^t=R_{e_ie_j}(\omega)$.\\
We remark that the mean Laplacian $\frac{1}{2}\bigl( \Delta +\tilde{\Delta}\bigr)$ sends symmetric double forms onto symmetric double forms and satisfies the following Weitzenb\"ock formula
\[\frac{1}{2}\bigl( \Delta(\omega)+\tilde{\Delta}(\omega)\bigr)=\nabla^*\nabla(\omega)-\frac{1}{2}\sum_{i<j}(\rho_{ij})_dR_{e_ie_j}(\omega),\]
where $(\rho_{ij})_d=(\rho_{ij})_l+(\rho_{ij})_r$ is the extension by derivations of the operator $\rho_{ij}$ to double forms as in section \ref{extensions-section}.\\ 
Choosing an orthonormal basis $(E_{\alpha})$ of $2$-forms that diagonalizes the curvature operator $R$, one has $R=\sum_{\alpha} \lambda_{\alpha} E_{\alpha}\otimes E_{\alpha}$ and hence can re-write the above Weitzenb\"ock formula as 
\[ \frac{1}{2}\bigl( \Delta(\omega)+\tilde{\Delta}(\omega)\bigr)=\nabla^*\nabla(\omega)-\frac{1}{2}  \sum_{\alpha}\lambda_{\alpha}\bigl(\rho(E_{\alpha})\bigr)_d \circ \bigl(\rho(E_{\alpha})\bigr)_d(\omega).\]
We remark that the term of order zero in the above formula is nothing but the curvature operator defining the Lichn\'erowicz Laplacian with $c=1/2$, see \cite[Section 1.143]{Besse} and \cite[Lemma 9.3.3]{Petersen-book}. We have therefore proved the following proposition
\begin{proposition}
The Lichn\'erowicz Laplacian $\Delta_L$ on general $(p,q)$ double forms coincides with the average of left and right Hodge Laplacians. Precisely, we have
\begin{equation}
\Delta_L=\frac{1}{2}\bigl( \Delta +\tilde{\Delta}\bigr)=\frac{1}{2}\bigl( D\delta+\delta D+\tilde{D}\tilde{\delta}+\tilde{\delta}\tilde{D}.\bigr)
\end{equation}
In particular, if $\omega$ is a  $(p,q)$ double form one has with respect to the $L^2$ norm of double forms
\begin{equation}
\langle \Delta_L(\omega),\omega\rangle =|D\omega|^2+|\delta\omega|^2+|\tilde{D}\omega|^2+|\tilde{\delta}\omega|^2,
\end{equation}
consequently, $\omega$ is $\Delta_L$-harmonic if and only if $D\omega=\delta\omega=\tilde{D}\omega=\tilde{\delta}\omega=0.$
\end{proposition}

Clearly, a $\Delta_L$-harmonic double form is automatically harmonic. However the converse is not generally true as shown in the following proposition
\begin{proposition}
\begin{enumerate}
\item A $(p,0)$ double form $\omega$ is $\Delta_L$-harmonic if and only if $\omega$ is a parallel double form (or equivalently, a parallel $p$-form).
\item A symmetric or skew-symmetric $(p,p)$ double form is $\Delta_L$-harmonic if and only if it is harmonic.
\end{enumerate}
\end{proposition}
\begin{proof}
For a $p$-form $\omega$ seen as a $(p,0)$ double form  one clearly has $D\omega=d\omega$, $\delta\omega$ is the divergence of $\omega$ as a differential $p$-form, $\tilde{\delta}(\omega)=0$ and $D(\omega^t)=\nabla\omega$. The last assertion can be proved as follows (we look to $\omega$ as a $p$-form)
\[D(\omega^t)=\sum_{i=1}^n(e^i\otimes 1).\nabla_{e_i}(1\otimes \omega)=\sum_{i=1}^ne^i\otimes \nabla_{e_i}(\omega)=\nabla\omega.\]
Therefore, $\omega$ is $\Delta_L$-harmonic if and only if $d\omega=\delta\omega=\nabla\omega=0$. This completes the proof of the first part.
Next,  in case $\omega$ is a symmetric or skew-symmetric $(p,p)$ double form one has
\[
\langle \Delta_L(\omega),\omega\rangle =\langle \Delta(\omega),\omega\rangle.\]
In fact, this follows directly from the fact that  $\langle \bigl(\Delta(\omega^t)\bigr)^t,\omega\rangle=\langle \Delta(\omega^t),\omega^t\rangle$. 
\end{proof}

The following corollary is an application of the above Proposition and Proposition \ref{Weit-sharp-form}
\begin{corollary}
Let $\mathcal{N}_p=\frac{g^{p-1}\Ric}{(p-1)!}-2 \frac{g^{p-2}R}{(p-2)!}$ for $p\geq 2$ and $\mathcal{N}_1=\Ric$. For a $(p,q)$ double form $\omega$, the Weitzenb\"ock formula corresponding to Lichn\'erowicz Laplacian on double forms takes the form
\begin{equation}
\Delta_L(\omega)=\nabla^*\nabla(\omega)+\frac{1}{2}\Bigl(\omega \circ \mathcal{N}_p+\mathcal{N}_q\circ \omega\Bigr) -\frac{1}{2}R\# \omega.
\end{equation}
\end{corollary}
\begin{proof}
It is not difficult to use formula (\ref{sharp-composition}) to see that $\bigl( R\# \omega^t\bigr)^t=R\# \omega.$ Consequently, for a $(p,q)$ double form one has
\[\bigl(\Weit(\omega^t)\bigr)^t=\Bigl(\omega^t \circ \mathcal{N}_q-\frac{1}{4} R\# \omega^t\Bigr)^t=\mathcal{N}_q\circ \omega -\frac{1}{4}R\# \omega.\]
\end{proof}
We list in the next proposition some useful properties of the Lichn\'erowicz Laplacian.
\begin{proposition}
Let $\omega$ be an arbitrary double form then the following identities hold
\begin{align}
\Delta_L(\str \omega)=&  \str \Delta_L( \omega)     ,  &    \Ric_L(\str \omega)&=\str \Ric_L( \omega) ,&\label{eq20}\\
\Delta_L(g^r \omega)=&  g^r \Delta_L( \omega)     ,  &    \Delta_L(\cc^r \omega)&=\cc^r \Delta_L( \omega) ,&\label{eq22}\\
\Delta_L(\omega^t)=&\bigl(\Delta_l(\omega)\bigr)^t , &   &    &    &\label{eq21}
\end{align}
If furthermore,  $\omega$ is a $(p,p)$ double form and has the orthogonal decomposition $\omega=\sum_{i=0}^pg^{i}\omega_{p-i}$ then the orthogonal decomposition of $\Delta_L(\omega)$ is given by
\begin{equation}\label{last} \Delta_L(\omega)=\sum_{i=0}^pg^{i}\Delta_L(\omega_{p-i}).\end{equation}
In particular, $\omega$ is $\Delta_L$-harmonic if and only if all the components $\omega_i$ are $\Delta_L$-harmonic for  $i=0,1,2,...,p.$
\end{proposition}
\noindent 
We bring to the attention of the reader that equations (\ref{eq22}) and (\ref{eq21}) were first proved by Ogawa, \cite{Ogawa}.
\begin{proof}
Proposition \ref{Hodge-properties} shows that $\Delta(\str \omega)=\str\Delta(\omega)$ the same property is true for $\widetilde{\Delta}$ and therefore for $\Delta_L$. The second property in equation (\ref{eq20}) results from the fact that $2\Ric_L=\Weit(\omega)+\bigl(\Weit(\omega^t)\bigr)^t$ and the second part of proposition \ref{Hodge-properties}.\\
 To prove the first identity in \ref{eq22}, we use equations (\ref{eq2}) and (\ref{eq6}) of this paper to show that
\begin{equation*}
\begin{split}
\Delta(g^r\omega)=&g^r\Delta\omega+r(-1)^r\bigl[\tD D\omega-D\tD\omega\bigr]g^{r-1},\\
\widetilde{\Delta}(g^r\omega)=&g^r\widetilde{\Delta}\omega+r(-1)^r\bigl[D\tD \omega-\tD D\omega\bigr]g^{r-1}.
\end{split}
\end{equation*}
Summing the above two equations yields the desired result. Next, applying  the identity $\cc^r=(-1)^{(p+q)(n+r+1)}\str g^r\str $ which results from equation (\ref{identity2}) and the above relations, one gets the following
$$\Delta_L(\cc^r \omega)=(-1)^{(p+q)(n+r+1)} \Delta_L(\str g^r\str \omega) =(-1)^{(p+q)(n+r+1)}\str g^r\str  \Delta_L( \omega)=\cc^r \Delta_L( \omega).$$
Equation (\ref{eq21}) follows at once from the identity $\Delta(\omega^t)=\bigl(\widetilde{\Delta}(\omega)\bigr)^t.$ The last equation (\ref{last}) and property follow directly from the above relations after remarking that $\cc \Delta_L( \omega_i)= \Delta_L(\cc \omega_i)=0$.
\end{proof}

\subsubsection{A vanishing theorem for the Lichn\'erowicz Laplacian}

We start by the following algebraic lemma
\begin{lemma}
Let $(E_{\alpha})$ be an orthonormal basis of $2$-forms on an Euclidean $n$-space $(V,g)$ such that  for each  $\alpha$, the rank of $E_{\alpha}$ is $\leq 2r$, for some fixed integer $r$. Let $N=\frac{n(n-1)}{2}$ and $\bigl(\rho(E_{\alpha})\bigr)_d \circ \bigl(\rho(E_{\alpha})\bigr)_d$be the symmetric operator operating on $(p,q)$ double forms as above,   Then  the following holds
\begin{flalign}
-&\frac{1}{2}\sum_{\alpha=1}^N\bigl(\rho(E_{\alpha})\bigr)_d \circ \bigl(\rho(E_{\alpha})\bigr)_d(\omega)=\frac{p(n-p)+q(n-q)+p+q}{2}\omega-g\cc \omega -\frac{1}{2}\Bigl(\tilde{\frak{S}}\frak{S}+\frak{S}\tilde{\frak{S}}\Bigr)(\omega).\\
 \text{For each} &\, \alpha, \text{the maximal  eigenvalue of }\,   -\frac{1}{2}\bigl(\rho(E_{\alpha})\bigr)_d \circ \bigl(\rho(E_{\alpha})\bigr)_d\,  \text{is} \,\leq \Bigl(2\Min\{p,r\}+2\Min\{q,r\}\Bigr).
\end{flalign}
Here, $\frak{S}$ and $ \tilde{\frak{S}} $ are respectively the  first Bianchi map and its adjoint map and $\cc$ is the contraction map in the space of double forms. 
\end{lemma}
\begin{proof}
We remark first that  $-\frac{1}{2}\sum_{\alpha=1}^N\bigl(\rho(E_{\alpha})\bigr)_d \circ \bigl(\rho(E_{\alpha})\bigr)_d(\omega)$ coincides with the Weitzenb\"ock curvature term  for the Lichn\'erowicz Laplacian with $c=1/2$ of the standard sphere with curvature $1$.  Equation (\ref{Weit-sphere}) asserts that for a $(p,q)$ double form $\omega$, the Weitzenb\"ock curvature $\Weit_{S^n}(\omega)$ of the standard sphere is given by
\[\Weit_{S^n}(\omega)=\bigl[ p(n-p)+q\bigr] \omega-g\cc \omega -\widetilde{\mathfrak{S}}\mathfrak{S}\omega.\]
Consequently, it is not difficult to conclude that
\[\Bigl(\Weit_{S^n}(\omega^t)\Bigr)^t=\bigl[ q(n-q)+p\bigr] \omega-g\cc \omega -\mathfrak{S}\widetilde{\mathfrak{S}}\omega.\]
The Weitzenb\"ock curvature term  in the formula defining the Lichn\'erowicz Laplacian with $c=1/2$  of the standard sphere with curvature $1$ is then the average of the above two terms, (see the above discussion). This completes the proof of the first part.\\
Next, we prove the second part. We proceed as in the proof of Proposition \ref{B-alpha-estimations} and use the same notations. The eigenvalues of the operator $-\bigl(\rho(E_{\alpha})\bigr)_d\circ \bigl(\rho(E_{\alpha})\bigr)_d$  operating on $(p,q)$ double forms are sums of $k+l$ terms of the form $\bigl( (\pm \lambda_{i_1}\pm ... \pm \lambda_{i_k})+(\pm \lambda_{j_1}\pm ... \pm \lambda_{j_l}) \bigr)^2$ with
$1\leq i_1<i_2<...<i_k\leq r, \,\, 1\leq j_1<j_2<...<j_l\leq r$, $k\leq \min\{p,r\}$, $l\leq \min\{q,r\}$ and $2r$ is the rank of the 2-form $E_{\alpha}$. The remaining eigenvalues are zero.\\
Using Newton's inequality and the fact that the eigenvalues of $\bigl(\rho(E_{\alpha})\bigr)_d\circ \bigl(\rho(E_{\alpha})\bigr)_d$ come in pairs, we show that
\begin{equation*}
\begin{split}
\Bigl[ (\pm \lambda_{i_1}\pm ... \pm \lambda_{i_k})+(\pm \lambda_{j_1}\pm ... \pm \lambda_{j_l}) \Bigr]^2
&\leq  (k+l)\Bigl(\lambda_{i_1}^2+...+\lambda_{i_k}^2+\lambda_{j_1}^2+...+\lambda_{j_l}^2\Bigr)\\
 \leq 2(k+l)(\lambda_1^2+&...+\lambda_r^2)=
(k+l)||\rho(E_{\alpha})||^2=2(k+l)||E_{\alpha}||^2=2k+2l.
\end{split}
\end{equation*}
\end{proof}

Let $\omega$ be a non-zero $(p,q)$ double form on an Euclidean $n$-space $(V,g)$, such that $p$ and $q$ are not both zero. Let $r$ be a positive real number, we define  $k_g(\omega,r)$ to be the real number  defined by
\begin{equation}k_g(\omega,r)=\frac{1}{2\bigl(\Min\{p,r\}+\Min\{q,r\}\bigr)}\Bigl[p(n-p+1)+q(n-q+1)-2\frac{||\cc \omega||^2}{||\omega||^2}-\frac{||\mathfrak{S}\omega||^2+||\tilde{\mathfrak{S}}\omega||^2}{||\omega||^2}\Bigr],\end{equation}
where $\cc$ (resp. $\mathfrak{S}, \tilde{\mathfrak{S}}$) are the contraction map (resp. the first Bianchi map and its adjoint map) of double forms.

As a straightforward consequence of the above Lemma and Lemma \ref{weighted-sum} we have

\begin{proposition}
Let $(E_{\alpha})$ be an orthonormal basis of $2$-forms on an Euclidean $n$-space $(V,g)$ such that  for each  $\alpha$, the rank of $E_{\alpha}$ is $\leq 2r$, for some fixed integer $r$. Let $\omega$ be a non-zero $(p,q)$ double form with $p+q\geq 1$,
$N=\frac{n(n-1)}{2}$ and  $\lambda_1\leq \lambda_2\leq ...\leq \lambda_N$ be given arbitrary real numbers. If $k$ is any integer such that $1\leq k\leq k_g(\omega,r)$ then
\[\sum_{\alpha=1}^{k}\lambda_{\alpha}>0\, (\text{resp.}\, \geq 0)\, \implies  -\langle\sum_{\alpha=1}^N\lambda_{\alpha}\bigl(\rho(E_{\alpha})\bigr)_d \circ \bigl(\rho(E_{\alpha})\bigr)_d(\omega),\omega\rangle >0 \,\, (\text{resp.}\,\, \geq 0).\]
\end{proposition}

Let now $(M,g)$ be compact Riemannian $n$-manifold, $r$ a constant positive real number and let $\omega$ be a $(p,q)$ double form on $M$ such that $p$ and $q$ are not both zero . We define the constant
\begin{equation}
k_g(\omega,r)=\inf_{x\in M}\{k(\omega_x,r):  \omega_x\not=0\}.
\end{equation}
 
A consequence of the above proposition and the Weitzenb\"ock formula is the following result

\begin{theorem}
Let $(M,g)$ be a closed and connected Riemannian manifold of dimension $n\geq 3$ and let $\omega$ be a $\Delta_L$-harmonic $(p,q)$ double form on $M$ such that $p+q\geq 1$.\\
 Suppose the Riemann curvature tensor $R$ has purity rank $\leq 2r$ and $R$ is $k$-positive (resp. $k$-nonnegative) with
$1\leq k\leq k_g(\omega,r)$, then  $\omega =0$  (resp. $\omega$ must be parallel). 
\end{theorem}
Next, we specify the previous theorem to the case of symmetric double forms.

\begin{theorem}\label{cw=0}
Let $(M,g)$ be a closed and connected Riemannian manifold of dimension $n\geq 3$ and let $\omega$ be a harmonic $(p,p)$ symmetric double form on $M$, $ 1\leq p\leq n/2$,  satisfying the first Bianchi identity such that its first contraction satisfies $\cc \omega=0$. 
 Suppose the Riemann curvature tensor $R$ has purity rank $\leq 2r$ and $R$ is $k$-positive (resp. $k$-nonnegative) with
$1\leq k\leq \frac{p(n-p+1)}{2\Min\{p,r\}}$, then  $\omega =0$  (resp. $\omega$ must be parallel). 
\end{theorem}

\begin{remark}
We remark that a $(p,p)$ double form such that $2p>n$ and satisfying $\cc \omega=0$ must be zero itself without the positivity and purity assumptions on $R$, see Proposition 2.1 in \cite{Labbi-Balkan}.
\end{remark}
\noindent
Next, we prove Theorem B stated in the introduction as a consequence of the above theorem.
\begin{thm-b}
Let $(M,g)$ be a closed and connected Riemannian manifold of dimension $n\geq 3$ and let $\omega$ be a harmonic $(p,p)$ symmetric double form on $M$ with $1\leq p\leq n/2$,  satisfying the first Bianchi identity. 
 Let $j$ be an integer such that $1\leq j\leq p$.
 If the Riemann curvature tensor $R$ is $\lfloor \frac{n-j+1}{2}\rfloor$-positive then  $\cc^{p-j}\omega =\lambda g^j$ for some constant $\lambda$. In particular,
\begin{itemize}
\item[a)] If $R$ is $\lfloor \frac{n-p+1}{2}\rfloor$-positive then  $\omega$ has constant sectional curvature.
\item[b)] If $R$ is $\lfloor \frac{n}{2}\rfloor$-positive then  $\cc^{p-1}\omega =\lambda g$ for some constant $\lambda$.
\end{itemize}
\end{thm-b}
\noindent
In the case where $p=2$ and $\omega=R$ is the Riemann tensor,  we recover the main result of Tachibana in \cite{Tachibana}. For $p=1$, we recover a classical result that a harmonic symmetric $(1,1)$ double form  on a Riemannian manifold with positive curvature operator (in fact
 ${\rm sec}>0$ is enough)  must be proportional to the metric tensor.
\begin{proof} 
Recall that \cite{Kulkarni, Labbi-double-forms} any $(p,p)$ double form $\omega$ has an orthogonal decomposition $\omega=\sum_{i=0}^pg^{i}\omega_{p-i}$, where $g^i$ are the exterior powers of the metric $(1,1)$ double form $g$ and the $\omega_k$ are   $(k,k)$ double forms such that $\cc \omega_k=0$. It turns out that if $\omega$ is symmetric and satisfies the first Bianchi identity then so are all the components $\omega_i$ for $1\leq i\leq p$, see \cite{Kulkarni}.  We apply the above theorem \ref{cw=0} to the different $\omega_k$ components  and get the vanishing of  $\omega_k =0$ for 
all $k$ within $1\leq k\leq j$. On the other hand,  the vanishing of the components $\omega_k$ for all $k$ such $1\leq k\leq j$ is equivalent to the condition $\cc^{p-j}\omega=\lambda g^j$, with $\lambda$   constant, see \cite[Proposition 2.2]{Labbi-Balkan}.
\end{proof}
\noindent
We claim that the following classes of Riemannian manifolds have their Gauss-Kronecker tensor $R^k$ harmonic and to which the above theorem applies.
\begin{enumerate}
\item $2k$-Thorpe manifolds \cite{Labbi-Balkan}, see also \cite{Kim, Labbi-JAUMS}. These are Riemannian manifolds of even dimension $n=2p\geq 4k$ that satisfy the self-duality condition
$\str\bigl(g^{p-2k}R^k\bigr)=g^{p-2k}R^k.$
These are generalizations of Einstein manifolds obtained for $k=1$
\item $2k$-Anti-Thorpe manifolds. These are Riemannian manifolds of even dimension $n=2p\geq 4k$ that satisfy the  condition
$\str\bigl(g^{p-2k}R^k\bigr)+g^{p-2k}R^k=\lambda g^p,$
where $\lambda$ is a constant. These are generalizations of conformally flat manifolds with constant scalar curvature obtained for $k=1$, see \cite[Theorem 5.8]{Labbi-double-forms}.
\end{enumerate}

\begin{proof}
First, we prove that the double form $g^{p-2k}R^k$ is harmonic under the first condition . It is clear that $D(g^{p-2k}R^k)=0$.  Property \ref{eq8} shows that
$$\delta(g^{p-2k}R^k)=\str D\str (g^{p-2k}R^k)=\str D (g^{p-2k}R^k)=0.$$
Therefore, using equation  (\ref{eq22}) we see that $0=\Delta(g^{p-2k}R^k)=\Delta_L(g^{p-2k}R^k)=g^{p-2k}\Delta_L(R^k)$. Proposition \ref{1-1} shows that $\Delta_L(R^k)=0$. The proof for the second class is completely similar.
\end{proof}

\begin{thm-c}
Let $(M,g)$ be a closed and connected Riemannian manifold of dimension $n\geq 3$. Let $\omega$ be a harmonic $(p,p)$ symmetric double form on $M$ with $2\leq p\leq n/2$,  satisfying the first and second Bianchi identities such that its first contraction satisfies $\cc \omega=\lambda g^{p-1}$ for some function $\lambda$ on $M$.
Suppose the Riemann curvature tensor $R$ has purity rank $\leq 2r$ and $R$ is $k$-positive (resp. $k$-nonnegative) with
$1\leq k\leq \frac{p(n-p+1)}{2\Min\{p,r\}}$, then the double form $\omega$ has constant sectional curvature (resp. $\omega$ must be parallel). 
\end{thm-c}

\begin{remark}
We remark that a $(p,p)$ double form such that $2p>n$ and satisfying $\cc \omega=\lambda g^{p-1}$ must be with constant sectional curvature without the positivity and purity assumptions on $R$, see Propositions 2.1 and 2.2 in \cite{Labbi-Balkan}.
\end{remark}
\begin{proof}
According to the above remark, the orthogonal decomposition of the double form $\omega$ takes the form
$$\omega=\omega_p+g^p\omega_0=\omega_p+\frac{\lambda}{p(n-p+1)}g^p.$$
Taking one contraction we easily get $\lambda=p(n-p)\omega_0.$
Next, since $D\omega=0$ then $\tilde{D}\omega=0$, and so Proposition \ref{Schur} shows that the function  $\lambda$ is constant over $M$. Therefore, one has
\[\delta\omega=-(\cc \tilde{D}+\tilde{D}\cc)\omega=0+\tilde{D}\cc\omega=\tilde{D}(\lambda g^{p-1})=0.\]
Consequently, $\omega$ and $\omega_p$ are both harmonic. In the case of strict $k$-positivity, the vanishing of $\omega_p$ follows at once after applying the above theorem to $\omega_p$ and hence $\omega$ has constant sectional curvature. In the case of $k$-non-negativity, the above theorem shows that $\langle \Weit(\omega_p),\omega_p\rangle\geq 0.$
Consequently, using  Corollary \ref{Weitz-gp} we obtain
$$\langle \Weit(\omega),\omega\rangle=\langle \Weit(\omega_p),\omega_p\rangle\geq 0.$$
Since $\omega$ is harmonic, the Weitzenb\"ock formula implies that $\omega$ must be parallel.
\end{proof}

Recall that a Riemannian manifold is said to be hyper $2k$-Einstein if its Riemann curvature tensor seen as a $(2,2)$ double form  satisfies $\cc R^k= \lambda g^{2k-1}$, see \cite{Labbi-Balkan}. The above theorem shows that 

\begin{cor-c}
Let $(M,g)$ be a closed and connected hyper $2k$-Einstein Riemannian $n$-manifold with $2\leq 2k\leq n/2$. Suppose the Riemann curvature tensor $R$ has purity rank $\leq 2r$ and let  $p$ be any integer such that $1\leq p\leq \frac{k(n-2k+1)}{\Min\{2k,r\}}$. 
If $R$ is $p$-positive (resp. $p$-nonnegative) then $M$ has constant $(2k)$-th Thorpe sectional curvature, that is $R^k={\mathrm constant}.g^{2k}$ (resp. $R^k$ must be parallel). 
\end{cor-c}
We remark that $\frac{n-2k+1}{2}\leq \frac{k(n-2k+1)}{\Min\{2k,r\}}$ for any value of $r$, therefore we have the same conclusion as in the above corollary under the assumption that the sum of the lowest $\lfloor \frac{n-2k+1}{2}\rfloor$ eigenvalues of the curvature operator $R$ is positive  (resp. nonnegative).
For $k=1$, we recover  a recent result of Petersen and Wink \cite[Theorem B]{Pet-Wink}, see also \cite{Bet-Good}.\\

Let $2\leq 2k\leq n$ and let $ R^k= \sum_{i=0}^n g^{2k-i}\omega_i$ be the orthogonal decomposition of the $R^k$. Since $DR^k=0$, it results from proposition \ref{proposition5.3} that if $\delta\omega_{2k}=0$ then $\omega_{2k}$ is harmonic. The above theorem then tell us the following
\begin{corollary}
Let $2\leq 2k\leq n/2$ and let $(M,g)$ be a closed and connected Riemannian $n$-manifold. 
Suppose the Riemann curvature tensor $R$ has purity rank $\leq 2r$ and let  $p$ be any integer such that $1\leq p\leq \frac{k(n-2k+1)}{\Min\{2k,r\}}$. 
Let $R^k$ be the $k$-th exterior power of the Riemann tensor $R$ and $ R^k= \sum_{i=0}^n g^{2k-i}\omega_i$ be its orthogonal decomposition. If the top component   $\omega_{2k}$ is free of divergence and $R$ is $p$-positive (resp. $p$-nonnegative) then $(M,g)$ is $k$-conformally flat (resp. $\omega_{2k}$ is parallel.)
\end{corollary}
Recall that $k$-conformally flat  means $\omega_{2k}=0$, see \cite{Nasu}. The above corollary generalizes Corollary 3.3 of \cite{Pet-Wink} obtained for $k=1$. We remark that in case $2k>n/2$ then $\omega_{2k}=0$ without any further assumptions.

\end{document}